\RequirePackage{fix-cm}
\documentclass[smallextended,numbook]{svjour3}       
\smartqed  
\usepackage{graphicx}
\usepackage{latexsym}
\usepackage{amssymb}
\usepackage{amsmath}
\usepackage{relsize}

\newcommand{\nc}{\newcommand}
\nc{\nt}{\newtheorem}
\nt{thm}{Theorem}[section]
\nt{cor}[thm]{Corollary}
\nt{prop}[thm]{Proposition}
\nt{obs}[thm]{Observation}
\nt{lem}[thm]{Lemma}
\nt{defn}[thm]{Definition}
\nt{exa}[thm]{Example}
\nt{rem}[thm]{Remark}
\nt{ass}[thm]{Assumption}
\nt{alg}[thm]{Algorithm}
\nt{con}[thm]{Conjecture}
\nc{\ip}[2]{\mbox{$\langle #1,#2 \rangle$}}
\nc{\finpf}{\hfill{$\Box$}\linespace}
\nc{\linespace}{\vspace{\baselineskip} \noindent}
\nc{\R}{{\bf R}}
\nc{\cl}{\mbox{\rm cl}\,}
\nc{\cls}{ \mbox{{\scriptsize {\rm cl}}}\,}
\nc{\conv}{\mbox{\rm conv}\,}
\nc{\rb}{\mbox{\rm rb}\,}
\nc{\ri}{\mbox{\rm ri}\,}
\nc{\inter}{\mbox{\rm int}\,}
\nc{\kernel}{\mbox{\rm ker}\,}
\nc{\bd}{\mbox{\rm bd}\,}
\nc{\spann}{\mbox{\rm span}\,}
\nc{\rint}{\mbox{\rm rint}\,}
\nc{\epi}{\mbox{\rm epi}\,}
\nc{\gph}{\mbox{\rm gph}\,}
\nc{\tr}{\mbox{\rm tr}\,}
\nc{\rge}{\mbox{\rm rge}\,}
\nc{\rgel}{\mbox{\rm {\scriptsize rge}}\,}
\nc{\sepi}{\mbox{\rm {\scriptsize epi}}\,}
\nc{\sbd}{\mbox{\rm {\scriptsize bd}}\,}
\nc{\dom}{\mbox{\rm dom}\,}
\nc{\supp}{\mbox{\rm supp}\,}
\nc{\lin}{\mbox{\rm lin}\,}
\nc{\detr}{\mbox{\rm det}\,}
\nc{\para}{\mbox{\rm par}\,}
\nc{\crit}{\mbox{\rm crit}\,}
\nc{\cone}{\mbox{\rm cone}\,}
\nc{\convv}{\mbox{\rm conv}\,}
\nc{\rank}{\mbox{\rm rank}\,}
\nc{\fix}{\mbox{\rm Fix}}
\nc{\mx}{\mbox{\rm mx}}
\nc{\mult}{\mbox{\rm mult}\,}
\nc{\clco}{\overline{\mbox{\rm co}}\,}

\newcommand{\lf}{\operatornamewithlimits{liminf}}

\newcommand{\argmin}{\operatornamewithlimits{argmin}}
\newcommand{\argmax}{\operatornamewithlimits{argmax}}

\newenvironment{myequation}{\begin{equation}}{\end{equation}}

\newenvironment{myeqnarray*}{\begin{eqnarray*}}{\end{eqnarray*}}
\nc{\bmye}{\begin{myequation}} \nc{\emye}{\end{myequation}}

\def\tto{\;{\lower 1pt \hbox{$\rightarrow$}}\kern -12pt
           \hbox{\raise 2.8pt \hbox{$\rightarrow$}}\;}

\usepackage{enumerate}
\usepackage{graphicx}
\usepackage{mathrsfs}

\usepackage{url}
\begin{document}

\title{Optimality, identifiability, and sensitivity\thanks{Work of D. Drusvyatskiy on this paper has been partially supported by the NDSEG grant from the Department of Defence. Work of A. S. Lewis has been supported in part by National Science Foundation Grant DMS-0806057 and by the US-Israel Binational Scientific Foundation Grant 2008261.}
}


\author{D. Drusvyatskiy         \and
        A. S. Lewis 
}


\institute{D. Drusvyatskiy \at
              School of Operations Research and Information Engineering,
    Cornell University,
    Ithaca, New York, USA; \\
              Tel.: (607) 255-4856\\
              Fax: (607) 255-9129\\
              \email{dd379@cornell.edu}\\
              {\tt http://people.orie.cornell.edu/dd379/}.           
           \and
           A. S. Lewis \at
              School of Operations Research and Information Engineering,
    Cornell University,
    Ithaca, New York, USA;\\
    Tel.: (607) 255-9147 \\
    Fax:  (607) 255-9129\\
    \email{aslewis@orie.cornell.edu}\\
    {\tt http://people.orie.cornell.edu/aslewis/}.
}

\date{Received: date / Accepted: date}

\maketitle

\begin{abstract}
Around a solution of an optimization problem, an ``identifiable'' subset of the feasible region is one containing all nearby solutions after small perturbations to the problem.  A quest for only the most essential ingredients of sensitivity analysis leads us to consider identifiable sets that are ``minimal''.  This new notion lays a broad and intuitive variational-analytic foundation for optimality conditions, sensitivity, and active set methods.
\end{abstract}

\section{Introduction}
Active set ideas permeate traditional nonlinear optimization.  Classical problems involve a list of smooth nonlinear constraints:  the active set for a particular feasible solution --- the collection of binding constraints at that point --- is crucial in first and second order optimality conditions, in sensitivity analysis, and for certain algorithms.
Contemporary interest in more general constraints (such as semidefiniteness) suggests a reappraisal. A very thorough modern study of sensitivity analysis in its full generality appears in \cite{Bon_Shap}.  Approaches more variational-analytic in flavor appear in texts such as \cite{imp}. Our aim here is rather different: to present a simple fresh approach, combining wide generality with mathematical elegance.

Our approach has its roots in the notion of an ``identifiable surface'' \cite{Wright}, and its precursors 
\cite{Dunn87,Calamai-More87,Burke-More88,Burke90,Ferris91,Al-Khayyal-Kyparisis91,Flam92}.  In essence, the idea is extremely simple:  given a critical point  $x$  for a function  $f$,  a set  $M$  is {\em identifiable} if any sequence of points approaching  $x$  that is approximately critical (meaning corresponding subgradients approach zero) must eventually lie in  $M$.  The terminology comes from the idea that an iterative algorithm that approximates  $x$  along with an approximate criticality certificate must ``identify''  $M$.  To take the classical example where  $f$  is a pointwise maximum of smooth functions, around any critical point  $x$,  assuming a natural constraint qualification, we can define  $M$  as those points with the same corresponding ``active set'' of functions attaining the maximum.

Identifiable sets  $M$  are useful computationally because the problem of minimizing the function  $f$  near the critical point  $x$  is equivalent to minimizing the restriction of  $f$ to  $M$,  which may be an easier problem, and because the identifiability property allows convergent algorithms to find  $M$ --- the motivation for active set methods.  We show moreover how $M$  is a natural tool for optimality conditions:  under reasonable conditions, quadratic growth of  $f$  around  $x$  is equivalent to quadratic growth on  $M$ --- a potentially easier condition to check.

Clearly the smaller the identifiable set  $M$, the more informative it is.  Ideal would be a ``locally minimal identifiable set''.  We note that such sets may fail to exist, even for finite convex functions  $f$.  However, when a minimal identifiable set  $M$  does exist, we show that it is both unique (locally), and central to sensitivity analysis:  it consists locally of all critical points of small linear perturbations to $f$.  We show furthermore that, under reasonable conditions, variational analysis of  $f$  simplifies because, locally, the graph of its subdifferential mapping is influenced only by the restriction of  $f$  to  $M$.  One appealing consequence is a close relationship between minimal identifiable sets and critical cones appearing in the study of variational inequalities.

The case when an identifiable set $M$  is in fact a manifold around the point  $x$  (as in the classical example above) is particularly interesting.  Remarkably, this case is equivalent to a powerful but seemingly stringent list of properties known as ``partial smoothness'' \cite{Lewis-active}, nondegeneracy and prox-regularity --- related work on ``$\cal{VU}$ algorithms'' and ``the fast track'' appears in \cite{MC03,MC04,MC05} and \cite{HS06,HS09}. By contrast, our approach here is to offer a concise mathematical development emphasizing how this important scenario is in fact very natural indeed.

The outline of the paper is as follows. In Section~\ref{sec:ident}, we introduce the notion of identifiability for arbitrary set-valued mappings. Then in Section~\ref{sec:spec_app}, we specialize this idea to subdifferential mappings, laying the foundation for the rest of the paper. Section~\ref{sec:examples} contains basic examples of identifiable sets, while Section~\ref{sec:calculus} establishes a calculus of identifiability, which could be skipped at first reading. Arriving at our main results in Section~\ref{sec:geo}, we study variational geometry of identifiable sets; this in particular allows us to establish a strong relationship between identifiable sets and critical cones in Section~\ref{sec:crit}. Finally in Section~\ref{sec:opt}, we consider optimality conditions in the context of identifiable sets, while in Section~\ref{sec:man} we establish a relationship between identifiable manifolds and partial smoothness --- one of our central original goals. In Section~\ref{sec:func}, results of Section~\ref{sec:geo} are generalized to a function setting.

\section{Identifiability in set-valued analysis}\label{sec:ident}
A {\em set-valued mapping} $G$ from $\R^n$ to $\R^m$, denoted by $G\colon\R^n\rightrightarrows\R^m$, is a mapping from $\R^n$ to the power set of $\R^m$. Thus for each  point $x\in\R^n$, $G(x)$ is a subset of $\R^m$. The {\em domain}, {\em graph}, and {\em range} of $G$ are defined to be 
\begin{align*}
\mbox{\rm dom}\, G&:=\{x\in\R^n:G(x)\neq\emptyset\},\\
\mbox{\rm gph}\, G&:=\{(x,y)\in\R^n\times\R^m:y\in G(x)\},\\
\rge G&:= \bigcup_{x\in\R^n} G(x),
\end{align*}
respectively. Observe that $\dom G$ and $\rge G$ are images of $\gph G$ under the projections $(x,y)\mapsto x$ and $(x,y)\mapsto y$, respectively. 

For two sets $M$ and $Q$ in $\R^n$, we will say that the inclusion $M\subset Q$ {\em holds locally around a point} $\bar{x}$, if there exists a neighborhood $U$ of $\bar{x}$ satisfying $M\cap U\subset Q\cap U$.

The key property that we explore in this work is that of finite identification.

\begin{defn}[Identifiable sets]
{\rm Consider a mapping $G\colon\R^n\rightrightarrows\R^m$.
We say that a subset $M\subset\R^n$ is {\em identifiable} at $\bar{x}$ for $\bar{v}$, where $\bar{v}\in G(\bar{x})$, if the inclusion
$$\gph G\subset M\times\R^n \textrm{ holds locally around } (\bar{x},\bar{v}).$$}
\end{defn}
Equivalently, a set $M$ is identifiable at $\bar{x}$ for $\bar{v}\in G(\bar{x})$ if for any sequence $(x_i,v_i)\to (\bar{x},\bar{v})$ in $\gph G$, the points $x_i$ must lie in $M$ for all sufficiently large indices $i$. Clearly $M$ is identifiable at $\bar{x}$ for $\bar{v}$ if and only if the same can be said of $M\cap\dom G$. Hence we will make light of the distinction between two such sets.

Clearly $\dom G$ is identifiable at $\bar{x}$ for $\bar{v}\in G(\bar{x})$. More generally, if $\bar{v}$ lies in the interior of some set $U$, then $G^{-1}(U)$ is identifiable at $\bar{x}$ for $\bar{v}$. Sometimes {\em all} identifiable subsets of $\dom G$ arise locally in this way. In particular, one can readily check that this is the case for any set-valued mapping $G$ satisfying $G^{-1}\circ G=\textrm{Id}$, an important example being the inverse of the projection map $G=P_Q^{-1}$ onto a nonempty, closed, convex set $Q$.



The ``smaller'' the set $M$ is, the more interesting and the more useful it becomes. 
Hence an immediate question arises. When is the identifiable set $M$ {\em locally minimal}, in the sense that for any other identifiable set $M'$ at $\bar{x}$ for $\bar{v}$, the inclusion $M\subset M'$ holds locally around $\bar{x}$? The following notion will be instrumental in addressing this question.

\begin{defn}[Necessary sets]
{\rm Consider a set-valued mapping $G\colon\R^n\rightrightarrows\R^m$, a point $\bar{x}\in\R^n$, and a vector $\bar{v}\in G(\bar{x})$. We say that a subset $M\subset\R^n$, containing $\bar{x}$, is {\em necessary} at $\bar{x}$ for $\bar{v}$ if the function $$x\mapsto d(\bar{v}, G(x)),$$ restricted to $M$, is continuous at $\bar{x}$. }
\end{defn}

Thus $M$ is necessary at $\bar{x}$ for $\bar{v}\in G(\bar{x})$ if for any sequence $x_i\to\bar{x}$ in $M$, there exists a sequence $v_i\in G(x_i)$ with $v_i\to\bar{v}$. The name ``necessary'' arises from the following simple observation.
\begin{lem}\label{lem:dir}
Consider a set-valued mapping $G\colon\R^n\rightrightarrows\R^m$, a point $\bar{x}\in\R^n$, and a vector $\bar{v}\in G(\bar{x})$. Let $M$ and $M'$ be two subsets of $\R^n$. Then the implication
\begin{equation*}
\left.
\begin{array}{c}
M \textrm{ is identifiable at $\bar{x}$ for $\bar{v}$}  \\
M' \textrm{ is necessary at $\bar{x}$ for $\bar{v}$}
\end{array}
\right\}
\Rightarrow
M'\subset M \textrm{ locally around } \bar{x},
\end{equation*}
holds.
\end{lem}

The following elementary characterization of locally minimal identifiable sets will be used extensively in the sequel, often without an explicit reference.
\begin{prop}[Characterizing locally minimal identifiability]\label{prop:loc_min} 
 {\ } \\ Consider a set-valued mapping $G\colon\R^n\rightrightarrows\R^m$ and a pair $(\bar{x},\bar{v})\in\gph G$. The following are equivalent.
\begin{enumerate}
\item $M$ is a locally minimal identifiable set at $\bar{x}$ for $\bar{v}$.
\item There exists neighborhood $V$ of $\bar{v}$ such that for any subneighborhood $W\subset V$ of $\bar{v}$, the representation 
$$M= G^{-1}(W) \textrm{ holds locally around } \bar{x}.$$
\item $M$ is a locally maximal necessary set at $\bar{x}$ for $\bar{v}$.
\item $M$ is identifiable and necessary at $\bar{x}$ for $\bar{v}$.
\end{enumerate}
\end{prop}
\begin{proof} 
$(1)\Rightarrow (2):$ Suppose $M$ is a locally minimal identifiable set at $\bar{x}$ for $\bar{v}$. Then there exists a neighborhood $V$ of $\bar{v}$ satisfying
$$G^{-1}(V)\subset M \textrm{ locally around } \bar{x}.$$
In turn, the set $G^{-1}(V)$ is identifiable at $\bar{x}$ for $\bar{v}$. Consequently, by local minimality of $M$ we deduce the equality $M=G^{-1}(V)$ locally around $\bar{x}$. Observe that the same argument is valid for any subneighborhood $W\subset V$ of $\bar{v}$. The result follows.

$(2)\Rightarrow (3):$ We first prove that $M$ is necessary at $\bar{x}$ for $\bar{v}$. Consider any sequence $x_i\to\bar{x}$ in $M$ and, for the sake of contradiction, suppose that there do not exist vectors $v_i\in G(x_i)$ with $v_i\to \bar{v}$. Then there exists a real number $\epsilon >0$ such that, restricting to a subsequence, we have $G(x_i)\cap B_{\epsilon}(\bar{v})=\emptyset$. Consequently, we obtain $x_i\notin G^{-1}(V\cap B_{\epsilon}(\bar{v}))$ for all large indices $i$. We conclude that $G^{-1}(V\cap B_{\epsilon}(\bar{v}))$ and $M$ do not coincide on any neighborhood of $\bar{x}$, which is a contradiction.
Hence $M$ is necessary at $\bar{x}$ for $\bar{v}$. 

Now consider any other necessary set $M'$ at $\bar{x}$ for $\bar{v}$ and a sequence $x_i\to\bar{x}$ in $M'$. By necessity of $M'$, there exist vectors $v_i\in G(x_i)$ converging to $\bar{v}$. By the defining property of $M$, we deduce $x_i\in M$ for all large indices $i$. Hence the inclusion $M'\subset M$ holds locally around $\bar{x}$. We conclude that $M$ is a locally maximal necessary set at $\bar{x}$ for $\bar{v}$.

$(3)\Rightarrow (4):$ Consider any sequence $(x_i,v_i)\to (\bar{x},\bar{v})$ in $\gph G$. Observe that 
the set $$M':=M\cup\bigcup_{i=1}^{\infty}\{x_i\},$$ is necessary at $\bar{x}$ for $\bar{v}$. Consequently, by local maximality of $M$, we have $x_i\in M$ for all large indices $i$. Thus $M$ is identifiable at $\bar{x}$ for $\bar{v}$.

$(4)\Rightarrow (1):$ This is immediate from Lemma~\ref{lem:dir}.
\end{proof}

\begin{rem}
{\rm
It is clear from Proposition~\ref{prop:loc_min} that whenever locally minimal identifiable sets exist, they are locally unique. That is, if $M_1$ and $M_2$ are both locally minimal identifiable sets at $\bar{x}$ for $\bar{v}\in G(\bar{x})$, then we have $M_1=M_2$ locally around $\bar{x}$. }
\end{rem}

The central goal in sensitivity analysis is to understand the behavior of solutions $x$, around $\bar{x}$, to the inclusion 
$$v\in G(x),$$ as $v$ varies near $\bar{v}$. Characterization $2$ of Proposition~\ref{prop:loc_min} shows that a locally minimal identifiable set at $\bar{x}$ for $\bar{v}$ is a locally minimal set that captures {\em all} the sensitivity information about the inclusion above. 

This characterization yields a constructive approach to finding locally minimal identifiable sets. Consider any open neighborhoods $V_1\supset V_2\supset V_3\supset\ldots,$ around $\bar{v}$ with the diameters of $V_i$ tending to zero. If the chain 
$$G^{-1}(V_1)\supset G^{-1}(V_2)\supset G^{-1}(V_3)\supset\ldots,$$
stabilizes, in the sense that for all large indices $i$ and $j$, we have $G^{-1}(V_i)=G^{-1}(V_j)$ locally around $\bar{x}$, then $G^{-1}(V_i)$ is a locally minimal identifiable set at $\bar{x}$ for $\bar{v}$, whenever $i$ is sufficiently large. Moreover, the locally minimal identifiable set at $\bar{x}$ for $\bar{v}$, if it exists, must arise in this way. 

The following example shows that indeed a set-valued mapping can easily fail to admit a locally minimal identifiable set.

\begin{exa}[Failure of existence]\label{exa:nonexist}
{\rm Consider the mapping $G\colon\R^2\rightrightarrows\R$, defined in polar coordinates, by
\begin{displaymath}
G(r,\theta) = \left\{
     \begin{array}{ll}
     |\theta| &\textrm{ if }  r\neq 0, ~\theta\in \left[-\pi,\pi\right], \\
  	 \left[-1,1\right] &\textrm{ if } r=0. \\
\end{array}
\right.
\end{displaymath}
Let $\bar{x}$ be the origin in $\R^2$ and $\bar{v}:=0\in G(\bar{x})$. Observe that for $\epsilon\to 0$, the preimages 
$$G^{-1}(-\epsilon,\epsilon)=\{(r,\theta)\in\R^2:G(r,\theta)\cap (-\epsilon,\epsilon)\neq\emptyset\}= \{(r,\theta):|\theta|<\epsilon\},$$ never coincide around $\bar{x}$. Consequently, there is no locally minimal identifiable set at $\bar{x}$ for $\bar{v}$. 
} 
\end{exa}

Notwithstanding the previous example, locally minimal identifiable sets do often exist. 
Before proceeding, we recall the following two notions of continuity for set-valued mappings.
\begin{defn}[Continuity]
{\rm Consider a set-valued mapping $F\colon\R^n\rightrightarrows\R^m$.
\begin{enumerate}
\item $F$ is {\em outer semicontinuous} at a point $\bar{x}\in\R^n$ if for any sequence of points $x_i\in\R^n$ converging to $\bar{x}$ and any sequence of vectors $v_i\in F(x_i)$ converging to $\bar{v}$, we must have $\bar{v}\in F(\bar{x})$.
\item $F$ is {\em inner semicontinuous} at $\bar{x}$ if for any sequence of points $x_i$ converging to $\bar{x}$ and any vector $\bar{v}\in F(\bar{x})$, there exist vectors $v_i\in F(x_i)$ converging to $\bar{v}$.
\end{enumerate}
If both properties hold, then we say that $F$ is {\em continuous} at $\bar{x}$. We will say that $F$ is inner-semicontinuous at $\bar{x}$, {\em relative} to a certain set $Q\subset\R^n$, if the condition above for inner-semicontinuity is satisfied for sequences $x_i\to\bar{x}$
in $Q$.  }
\end{defn}

Clearly inner-semicontinuous mappings always admit locally minimal identifiable sets.
 
\begin{prop}[Identifiability under continuity]\label{prop:continuity} {\ } \\ 
Consider a set-valued mapping $G\colon\R^n\rightrightarrows\R^m$ that is inner semicontinuous, relative to $\dom G$, at a point $\bar{x}\in\dom G$. Then $\dom G$ is a locally minimal identifiable set at $\bar{x}$ for any vector $\bar{v}\in G(\bar{x})$. 
\end{prop}

More interesting examples can be constructed by taking pointwise unions of maps admitting locally minimal identifiable sets. 

\begin{prop}[Pointwise union]\label{prop:union}
Consider a finite collection of outer-semicontinuous mappings, 
$G_i\colon\R^n\rightrightarrows\R^m$, for $i=1,\ldots,k.$ Define the pointwise union mapping 
$G\colon\R^n\rightrightarrows\R^m$
to be $$G(x)=\bigcup^m_{i=1} G_i(x).$$
Fix a point $\bar{x}\in\R^n$ and a vector $\bar{v}\in G(\bar{x})$, and suppose that for each index $i$, satisfying $\bar{v}\in G_i(\bar{x})$, there exists a locally minimal identifiable set $M_i$ (with respect to $G_i$) at $\bar{x}$ for $\bar{v}$. Then the set $$M:=\bigcup_{i: \bar{v}\in G_i(\bar{x})} M_i,$$ is a locally minimal identifiable set (with respect to $G$) at $\bar{x}$ for $\bar{v}$.  
\end{prop}
\begin{proof} 
This readily follows from Proposition~\ref{prop:loc_min}. 
\end{proof}

In particular, locally minimal identifiable sets exist for {\em piecewise polyhedral} mappings. These are those mappings whose graphs can be decomposed into a union of finitely many convex polyhedra.

\begin{exa}[Piecewise polyhedral mappings]\label{exa:poly} {\ } \\ 
{\rm Consider a piecewise polyhedral mapping $G\colon\R^n\rightrightarrows\R^m$, where $\gph G=\bigcup_{i=1}^k V_i$ and $V_i\subset\R^n$ are convex polyhedral sets.
It is easy to check that set-valued mappings whose graphs are convex polyhedral are inner-semicontinuous on their domains.
Fix a point $\bar{x}\in \R^n$ and a vector $\bar{v}\in G(\bar{x})$, and let $\pi\colon\R^n\times\R^m\to\R^n$ be the canonical projection onto $\R^n$. 
Consequently, by Propositions~\ref{prop:continuity} and \ref{prop:union}, the set
$$\bigcup_{i:(\bar{x},\bar{v})\in V_i} \pi(V_i)$$ is a locally minimal identifiable set at $\bar{x}$ for $\bar{v}$. }
\end{exa}



For the remainder of the current work, we will be investigating the notion of identifiability in the context of the workhorse of variational analysis, the {\em subdifferential} mapping (Definition~\ref{defn:subdiff}).

\section{Identifiability in variational analysis}\label{sec:spec_app}
\subsection{Preliminaries from variational analysis}
In this section, we briefly summarize some of the fundamental tools used in variational analysis and nonsmooth optimization.
We refer the reader to the monographs Borwein-Zhu \cite{Borwein-Zhu}, Clarke-Ledyaev-Stern-Wolenski \cite{CLSW}, Mordukhovich \cite{Mord_1,Mord_2}, and Rockafellar-Wets \cite{VA}, for more details.  Unless otherwise stated, we follow the terminology and notation of \cite{VA}.

The functions that we will be considering will take their values in the extended real line $\overline{\R}:=\R\cup\{-\infty,\infty\}$. We say that an extended-real-valued function is proper if it is never $\{-\infty\}$ and is not always $\{+\infty\}$. 

For a function $f\colon\R^n\rightarrow\overline{\R}$, the {\em domain} of $f$ is $$\mbox{\rm dom}\, f:=\{x\in\R^n: f(x)<+\infty\},$$ and the {\em epigraph} of $f$ is $$\mbox{\rm epi}\, f:= \{(x,r)\in\R^n\times\R: r\geq f(x)\}.$$ We will say that $f$ is {\em lower-semicontinuous} ({\em lsc} for short) at a point $\bar{x}$ provided that the inequality $\lf_{x\to\bar{x}}f(x)\geq f(\bar{x})$ holds. If $f$ is lower-semicontinuous at every point, then we will simply say that $f$ is lower-semicontinuous.

Throughout this work, we will only use Euclidean norms. Hence for a point $x\in\R^n$, the symbol $|x|$ will denote the standard Euclidean norm of $x$.
Given a set $Q\subset\R^n$ and a point $\bar{x}\in Q$, 
we let $o(|x-\bar{x}|)$ for $x\in Q$ be shorthand for a function that satisfies $\frac{o(|x-\bar{x}|)}{|x-\bar{x}|}\rightarrow 0$ whenever $x\stackrel{Q}{\rightarrow} \bar{x}$ with $x\neq\bar{x}$.

For a set $Q\subset\R^n$ and a point $x\in\R^n$, the distance of $x$ from $Q$ is   $$d_Q(x):=\inf_{y\in Q}|x-y|,$$ and the projection of $x$ onto $Q$ is $$P_Q(x):=\{y\in Q:|x-y|=d_Q(x)\}.$$

Normal cones are fundamental in variational geometry.
The most intuitive type of a normal cone arises from the metric projection.
\begin{defn}[Proximal normals]
{\rm
Consider a set $Q\subset\R^n$ and a point $\bar{x}\in Q$. The {\em proximal normal cone} to $Q$ at $\bar x$, denoted
$N^{P}_Q(\bar x)$, consists of all vectors $v \in \R^n$ such that $\bar{x}\in P_Q(\bar{x}+\frac{1}{r}v)$ for some $r>0$. In this case, $\{\bar{x}\}=P_Q(\bar{x}+\frac{1}{r'}v)$ for each real number $r'>r$.}
\end{defn}

Geometrically, a vector $v\neq 0$ is a proximal normal to $Q$ at $\bar{x}$ precisely when there exists a ball touching $Q$ at $\bar{x}$ such that $v$ points from $\bar{x}$ towards the center of the ball. Furthermore, this condition amounts to 
$$\langle v,x-\bar{x} \rangle \leq O(|x-\bar{x}|^2) ~~\textrm{ as } x\to\bar{x} \textrm{ in } Q.$$

Relaxing the inequality above, one obtains the following notion.

\begin{defn}[Frech\'{e}t normals]
{\rm Consider a set $Q\subset\R^n$ and a point $\bar{x}\in Q$. The {\em Frech\'{e}t normal cone} to $Q$ at $\bar x$, denoted
$\hat N_Q(\bar x)$, consists of all vectors $v \in \R^n$ such that $$\langle v,x-\bar{x} \rangle \leq o(|x-\bar{x}|) \textrm{ as }x\to\bar{x} \textrm{ in } Q.$$
}
\end{defn}
Note that both $N^P_Q(\bar{x})$ and $\hat{N}_Q(\bar{x})$ are convex cones, while $\hat{N}_Q(\bar{x})$ is also closed. 

For a set $Q\in\R^n$, the set-valued mapping $x\mapsto \hat{N}_Q(x)$ is not outer-semicontinuous, and hence is not robust relative to perturbations in $x$. To correct for that, the following definition is introduced.
\begin{defn}[Limiting normals]
{\rm Consider a set $Q\subset\R^n$ and a point $\bar{x}\in Q$.  The {\em limiting normal cone} to $Q$ at $\bar{x}$, denoted $N_Q(\bar{x})$, consists of all vectors $v\in\R^n$ such that there are sequences $x_i\stackrel{Q}{\rightarrow} \bar{x}$ and $v_i\rightarrow v$ with $v_i\in\hat{N}_Q(x_i)$.}
\end{defn}

The limiting normal cone, as defined above, consists of limits of nearby Frech\'{e}t normals. In fact, the same object arises if we only allow limits of nearby proximal normals. An important and favorable situation arises when the Frech\'{e}t and limiting constructions coincide.

\begin{defn}[Clarke regularity of sets]
{\rm A set $Q\subset\R^n$ is said to be {\em Clarke regular} at a point $\bar{x}\in Q$ if it is locally closed at $\bar{x}$ and every limiting normal vector to $Q$ at $\bar{x}$ is a Frech\'{e}t normal vector, that is $N_Q(\bar{x})=\hat{N}_Q(\bar{x})$.}
\end{defn}

We can study variational properties of functions by means of normal cones to their epigraphs. 
\begin{defn}[Subdifferentials]\label{defn:subdiff}
{\rm Consider a function $f\colon\R^n\rightarrow\overline{\R}$ and a point $\bar{x}\in\R^n$ where $f$ is finite. The {\em proximal}, {\em Frech\'{e}t}, and {\em limiting subdifferentials} of $f$ at $\bar{x}$, respectively, are defined by
\begin{align*}
\partial_P f(\bar{x})&= \{v\in\R^n: (v,-1)\in N^{P}_{\mbox{{\scriptsize {\rm epi}}}\, f}(\bar{x},f(\bar{x}))\},\\
\hat{\partial}f(\bar{x})&= \{v\in\R^n: (v,-1)\in \hat{N}_{\mbox{{\scriptsize {\rm epi}}}\, f}(\bar{x},f(\bar{x}))\},\\
\partial f(\bar{x})&= \{v\in\R^n: (v,-1)\in N_{\mbox{{\scriptsize {\rm epi}}}\, f}(\bar{x},f(\bar{x}))\},
\end{align*}
while the {\em horizon subdifferential} is defined by 
$$\partial^{\infty} f(\bar{x})= \{v\in\R^n: (v,0)\in N_{\mbox{{\scriptsize {\rm epi}}}\, f}(\bar{x},f(\bar{x}))\}.$$
}
\end{defn}
For $\bar{x}$ such that $f(\bar{x})$ is not finite, we follow the convention that $\partial_{P} f(\bar{x})=\hat{\partial}f(\bar{x})=\partial f(\bar{x})=\partial^{\infty}f(\bar{x})=\emptyset$.

For convex functions $f$, the subdifferentials $\partial_p f$, $\hat{\partial} f$, and $\partial f$ reduce to the classical convex subdifferential, while for smooth $f$ they coincide with the gradient mapping $\nabla f$. See for example \cite[Exercise 8.8]{VA}. In this sense, these three subdifferentials generalize the classical gradient. The horizon subdifferential plays an entirely different role; it records horizontal normals to the epigraph of the function and is instrumental in establishing subdifferential calculus rules. See \cite[Theorem 10.6]{VA}.

\begin{defn}[Clarke regularity of functions]
{\rm A function $f\colon\R^n\to\overline{\R}$ is said to be {\em Clarke regular} at a point $\bar{x}\in \R^n$ if the epigraph $\epi f$ is Clarke regular at $(\bar{x},f(\bar{x}))$.}
\end{defn}

Given any set $Q\subset\R^n$ and a mapping $F\colon Q\to \widetilde{Q}$, where $\widetilde{Q}\subset\R^m$, we say that $F$ is ${\bf C}^p$-{\em smooth} $(p\geq 1)$ if for each point $\bar{x}\in Q$, there is a neighborhood $U$ of $\bar{x}$ and a ${\bf C}^p$ mapping $\hat{F}\colon \R^n\to\R^m$ that agrees with $F$ on $Q\cap U$. 

\begin{defn} [Smooth Manifolds]
{\rm We say that a set $M\subset{\R}^n$ is a ${\bf C}^p$-{\em submanifold} of dimension $r$ if for each point $\bar{x}\in M$, there is an open neighborhood $U$ around $\bar{x}$ and a function $F\colon U\to\R^{n-r}$ that is ${\bf C}^p$-smooth with $\nabla F(\bar{x})$ of full rank and satisfying 
$M\cap U=\{x\in U: F(x)=0\}$. In this case, we call $F$ a {\em local defining function} for $M$ around $\bar{x}$.
}
\end{defn}

A good source on smooth manifold theory is \cite{Lee}.

\begin{thm}[{\cite[Example 6.8]{VA}}]\label{thm:clarke_man}
Any ${\bf C}^1$-manifold $M$ is Clarke regular at every point $x\in M$ and the normal cone $N_M(x)$ is equal to the normal space to $M$ at $x$, in the sense of differential geometry.
\end{thm}

We will have occasion to use the following simple result. 
\begin{prop}\label{prop:help}
Consider a set $M\subset \R^n$ and a function $f\colon\R^n\to\overline{\R}$ that is finite-valued and ${\bf C}^{1}$-smooth on $M$.
Then, at any point $\bar{x}\in M$, we have $$\hat{\partial} f(\bar{x})\subset \nabla g(\bar{x})+\hat{N}_M(\bar{x}),$$ where $g\colon\R^n\to\R$ is any ${\bf C}^1$-smooth function agreeing with $f$ on a neighborhood of $\bar{x}$ in $M$.
\end{prop}
\begin{proof}
Define a function $h:\R^n\to\overline{\R}$ agreeing with $f$ on $M$ and equalling plus infinity elsewhere.
We successively deduce
$$\hat{\partial} f(\bar{x})\subset\hat{\partial} h(\bar{x})=\hat{\partial} (g(\cdot)+\delta_M(\cdot))(\bar{x})=\nabla g(\bar{x})+\hat{N}_M(\bar{x}),$$ as we need to show.
\end{proof}

For a set $Q\subset\R^n$, we let $\delta_Q$ denote a function that is $0$ on $Q$ and $+\infty$ elsewhere; we call $\delta_Q$ the {\em indicator function} of $Q$.
Then for any point $\bar{x}\in Q$, we have $N^{P}_Q(\bar{x})=\partial_P \delta_Q(\bar{x})$, $\hat{N}_Q(\bar{x})=\hat{\partial} \delta_Q(\bar{x})$ and $N_Q(\bar{x})=\partial \delta_Q(\bar{x})$.

\subsection{Identifiability in variational analysis}
We are now ready to define the appropriate notion of identifiability in the context of optimization.
\begin{defn}[Identifiability for functions] {\ } \\ 
{\rm Consider a function $f\colon\R^n\to\overline{\R}$, a point $\bar{x}\in\R^n$, and a subgradient $\bar{v}\in\partial f(\bar{x})$. A set $M\subset\R^n$ is {\em identifiable at} $\bar{x}$ {\em for} $\bar{v}$ if for any sequences $(x_i,f(x_i),v_i)\to(\bar{x},f(\bar{x}),\bar{v})$, with $v_i\in\partial f(x_i)$, the points $x_i$ must all lie in $M$ for all sufficiently large indices $i$.}
\end{defn}

The definition above can be interpreted in the sense of Section~\ref{sec:ident}. 
Indeed, consider a function $f\colon\R^n\to\overline{\R}$ and a subgradient $\bar{v}\in\partial f(\bar{x})$, for some point $\bar{x}\in\R^n$.
Define the set-valued mapping $$G\colon\R^n\rightrightarrows\R\times \R^n,$$
$$x\mapsto \{f(x)\}\times \partial f(x).$$ 
Then $M$ is identifiable (relative to $f$) at $\bar{x}$ for $\bar{v}$ if and only if
it is identifiable (relative to $G$) at $\bar{x}$ for the vector $(f(\bar{x}),\bar{v})$. Here, we have to work with the mapping $G$, rather than the subdifferential mapping $\partial f$ directly, so as to facilitate coherence between normal cone mappings and subdifferential mappings via epigraphical geometry. (See Proposition~\ref{prop:coherence}.)
This slight annoyance can be avoided whenever $f$ is subdifferentially continuous at $\bar{x}$ for $\bar{v}$.
\begin{defn}[Subdifferential continuity]
{\rm A function $f\colon\R^n\to\overline{\R}$ is {\em subdifferentially continuous at} $\bar{x}$ {\em for} $\bar{v}\in\partial f(\bar{x})$ if for any sequences $x_i\to\bar{x}$ and $v_i\to\bar{v}$, with $v_i\in\partial f(x_i)$, it must be the case that $f(x_i)\to f(\bar{x})$.}
\end{defn}

Subdifferential continuity of a function $f$ at $\bar{x}$ for $\bar{v}$ was introduced in \cite[Definition 1.14]{prox_reg}, and it amounts to requiring the function $(x,v)\mapsto f(x)$, restricted to $\gph \partial f$, to be continuous in the usual sense at the point $(\bar{x},\bar{v})$. In particular, any lower-semicontinuous convex function is subdifferentially continuous \cite[Example 13.30]{VA}. 

Similarly, we define necessary sets as follows.
\begin{defn}[Necessity for functions]
{\rm Consider a function $f\colon\R^n\to\overline{\R}$. A set $M\subset\R^n$ is {\em necessary at} $\bar{x}$ {\em for} $\bar{v}\in\partial f(\bar{x})$ if both the function $f$ and the mapping $$x\mapsto d(\bar{v}, \partial f(x)),$$ restricted to $M$, are continuous at $\bar{x}$. }
\end{defn}

Specializing the characterization in Proposition~\ref{prop:loc_min} to this setting, we obtain the following.
\begin{prop}[Characterizing locally minimal identifiability]\label{prop:charsub} {\ } \\  
Consider a function $f\colon\R^n\to\overline{\R}$, a point $\bar{x}\in\R^n$, and a subgradient $\bar{v}\in \partial f(\bar{x})$. Then the following are equivalent.
\begin{enumerate}
\item $M$ is a locally minimal identifiable set at $\bar{x}$ for $\bar{v}$,
\item There exists a neighborhood $V$ of $\bar{v}$ and a real number $\epsilon >0$ such that for any subneighborhood $W\subset V$ of $\bar{v}$ and a real number $0<\epsilon'<\epsilon$, the presentation
$$M=(\partial f)^{-1}(W)\cap\{x\in \R^n: |f(x)- f(\bar{x})|<\epsilon'\} \textrm{ holds locally around } \bar{x}.$$
\item $M$ is a locally maximal necessary set at $\bar{x}$ for $\bar{v}$.
\item $M$ is identifiable and necessary at $\bar{x}$ for $\bar{v}$
\end{enumerate}
\end{prop}

\begin{defn}[Identifiability for sets]
{\rm Given a set $Q\subset\R^n$, we will say that a subset $M\subset Q$ is {\em identifiable} (relative to $Q$) at $\bar{x}$ for $\bar{v}\in N_Q(\bar{x})$ if $M$ is identifiable (relative to $\delta_Q$) at $\bar{x}$ for $\bar{v}\in \partial\delta_Q(\bar{x})$. Analogous conventions will hold for necessary sets and locally minimal identifiable sets.}
\end{defn}

It is instructive to observe the relationship between identifiability and the metric projection in presence of convexity.
\begin{prop}[Identifiability for convex sets] {\ } \\ 
Consider a closed, convex set $Q$ and a subset $M\subset Q$. Let $\bar{x}\in M$ and $\bar{v}\in N_Q(\bar{x})$. Then the following are equivalent.
\begin{enumerate}
\item\label{it:usual} $M$ is identifiable (relative to $Q$) at $\bar{x}$ for $\bar{v}$.
\item\label{it:proj} $M$ is identifiable (relative to $P_Q^{-1}$) at $\bar{x}$ for $\bar{x}+\bar{v}$.
\end{enumerate}  
Analogous equivalence holds for necessary sets.
\end{prop}
\begin{proof}
Suppose that $M$ is identifiable (relative to $Q$) at $\bar{x}$ for $\bar{v}$. Consider a sequence $(x_i,y_i)\to(\bar{x},\bar{x}+\bar{v})$ in $\gph P_Q^{-1}$. Observe $x_i= P_Q(y_i)$ and the sequence $y_i-x_i\in N_Q(x_i)$ converges to $\bar{v}$. Consequently, the points $x_i$ all eventually lie in $M$.

Conversely suppose that $M$ is identifiable (relative to $P_Q^{-1}$) at $\bar{x}$ for $\bar{x}+\bar{v}$. Consider a sequence $(x_i,v_i)\to(\bar{x},\bar{v})$ in $\gph N_Q$. Then the sequence $(x_i,x_i+v_i)\in \gph P_Q^{-1}$ converges to $(\bar{x},\bar{x}+\bar{v})$. Consequently, we have $x_i\in M$ for all large $i$.

We leave the verification of the analogous equivalence for necessary sets to the reader.
\end{proof}

Thus a subset $M$ of a closed, convex set $Q$ is identifiable at $\bar{x}$ for $\bar{v}\in N_Q(\bar{x})$ if and only if the equality, $P_Q=P_M$, \textrm{ holds locally around } $\bar{x}+\bar{v}$.

The following simple proposition establishes epigraphical coherence, alluded to above, between normal cone mappings and subdifferential mappings in the context of identifiability.
\begin{prop}[Epigraphical coherence]\label{prop:coherence} {\ } \\ 
Consider a function $f\colon\R^n\to\overline{\R}$ and a subgradient $\bar{v}\in\partial f(\bar{x})$, for some point $\bar{x}\in\R^n$. Then $M\subset\dom f$ is an identifiable set (relative to $f$) at $\bar{x}$ for $\bar{v}$ if and only if $\gph f\big|_M$ is an identifiable set (relative to $\epi f$) at $(\bar{x},f(\bar{x}))$ for $(\bar{v},-1)$. Analogous statements hold for necessary, and consequently for locally minimal identifiable sets.
\end{prop}

\section{Basic Examples}\label{sec:examples}
In this section, we present some basic examples of identifiable sets. 
\begin{exa}[Smooth functions]
{\rm
If $U$ is an open set containing $\bar{x}$ and $f\colon U\to\R$ is ${\bf C}^1$-smooth. Then $U$ is a locally minimal identifiable set at $\bar{x}$ for $\nabla f(\bar{x})$.}
\end{exa}  

\begin{exa}[Smooth manifolds]
{\rm If $M$ is a ${\bf C}^1$ manifold, the $M$ is a locally minimal identifiable set at any $x\in M$ for any $v\in N_M(x)$. This follows immediately by observing that the normal cone mapping $x\mapsto N_M(x)$ is inner-semicontinuous on $M$.}
\end{exa}
  
We define the {\em support} of any vector $v\in\R^n$, denoted $\supp v$, to be the set consisting of all indices $i\in\{1,\ldots,n\}$ such that $v_i\neq 0$. The {\em rank} of $v$, denoted $\rank v$, is then the size of the support $\supp v$.  
  
\begin{exa}[Convex polyhedra]\label{exa:polyset}
{\rm Let $Q\subset\R^n$ be a convex polyhedron. Example~\ref{exa:poly} shows that $M:=N_Q^{-1}(\bar{v})$ (equivalently, $M:=\argmax_{x\in Q} \langle \bar{v},x \rangle$) is a locally minimal identifiable set at $\bar{x}$ for $\bar{v}$. 

More concretely, suppose that $Q$ has the representation
\begin{equation}\label{eqn:const_rep}
Q=\{x\in\R^n:\langle a_i,x \rangle\leq b_i \textrm{ for all }i\in I\},
\end{equation}
for some index set $I=\{1,\ldots,m\}$ and vectors $a_1,\ldots,a_m\in\R^n$ and $b\in\R^m$. For any point $x\in\R^n$, define the active index set $$I(x):=\{i\in I: \langle a_i,x \rangle =b_i\}.$$
Then we have $N_Q(x)=\cone\{a_i: i\in I(x)\}$ and consequently there exist multipliers $\bar{\lambda}\in\R^{m}_{+}$ with $\supp \bar{\lambda}\subset I(\bar{x})$ satisfying 
$\bar{v}=\sum_{i\in I(\bar{x})}\bar{\lambda}_i a_i$.
Hence for any point $y\in Q$, the equivalence
\begin{align*}
y\in M &\Longleftrightarrow \Big\langle \sum_{i\in I(\bar{x})}\bar{\lambda}_i a_i,y \Big\rangle=\Big\langle \sum_{i\in I(\bar{x})}\bar{\lambda}_i a_i,\bar{x} \Big\rangle \\ &\Longleftrightarrow  \sum_{i\in I(\bar{x})}\bar{\lambda}_i [\langle a_i,y\rangle-b_i ]=0,
\end{align*}
holds. We deduce that $M$ has the alternate description $$M=\{x\in Q:\supp \bar{\lambda} \subset I(x)\}.$$
We should note that under a strict complementarity condition, $\bar{v}\in\ri N_Q(\bar{x})$, we may choose $\bar{\lambda}$ with $\supp \bar{\lambda}=I(\bar{x})$. Then $M$ would consist of all points $x\in Q$, whose active index set $I(x)$ coincides with $I(\bar{x})$. It is then standard to check that $M$ coincides with an affine subspace locally around $\bar{x}$.
}
\end{exa}

\begin{exa}[Polyhedral functions]\label{exa:conv_poly}
{\rm
Analogously, we may analyse a convex polyhedral function $f\colon\R^n\to\overline{\R}$, a function whose epigraph is a convex polyhedron.
To be more precise, we may express $f$ as 
\begin{displaymath}
f(x) = \left\{
     \begin{array}{ll}
     \max_{i\in I}\{\langle a_i,x \rangle + b_i\} &\textrm{whenever } \langle c_j,x\rangle\leq d_j \textrm{ for all } j\in J,\\
  	 \infty &\textrm{otherwise},\\
\end{array}
\right.
\end{displaymath}
for some index sets $I=\{1,\ldots,m\}$ and $J=\{1,\ldots,k\}$, vectors $a_i,c_j\in\R^n$, and real numbers $b_i,d_j$ for $i\in I$ and $j\in J$. For any point $x\in\R^n$, define the active index sets 
\begin{align*}
I(x)&=\{i\in I: \langle a_i,x \rangle + b_i=f(x)\},\\
J(x)&=\{j\in J: \langle c_j,x \rangle =d_j\}.
\end{align*}
A straightforward computation shows 
$$\partial f(x)=\conv\{a_i:i\in I(x)\}+\cone\{c_j: j\in J(x)\}.$$
Consider a pair $(\bar{x},\bar{v})\in\gph \partial f$. Then there exist multipliers $(\bar{\lambda},\bar{\mu})\in\R^{m}_{+}\times\R^{k}_{+}$ satisfying 
$$\bar{v}=\sum_{i\in I}\bar{\lambda}_i a_i +\sum_{j\in J}\bar{\mu}_j c_j,$$
with $\sum_{i\in I}\bar{\lambda}_i =1$, $\supp \bar{\lambda}\subset I(\bar{x})$, and $\supp \bar{\mu} \subset J(\bar{x})$. 
Applying the same argument as in Example~\ref{exa:polyset} to $\epi f$, we deduce that the set
$$M=\{x\in\dom f:\supp \bar{\lambda} \subset I(x), ~\supp \bar{\mu} \subset J(x)\},$$ 
is a locally minimal identifiable set at $\bar{x}$ for $\bar{v}$. 
Again we should note that a particularly nice situation occurs under a strict complementarity condition, $\bar{v} \in \ri \partial f(\bar{x})$. In this case there exist multipliers $(\bar{\lambda}, \bar{\mu})$ so that $\supp \bar{\lambda}=I(\bar{x})$ and $\supp \bar{\mu}=J(\bar{x})$, and then $M$ coincides with an affine subspace locally around $\bar{x}$. 
}
\end{exa}

\begin{exa}[Maximum function]\label{exa:max}
{\rm
In particular, consider the maximum function $\mx\colon\R^n\to\R$, defined by 
$$\mx(x):=\max\{x_1,\ldots,x_n\}.$$ 
Then given a point $\bar{x}$ and a vector $\bar{v} \in\partial (\mx)(\bar{x})$, the set $M=\{x\in\R^n: \supp \bar{v}\subset I(x)\}$, where $$I(x):=\{i:x_i=\mx(x)\},$$
is a locally minimal identifiable set at $\bar{x}$ for $\bar{v}$. 
Alternatively, $M$ admits the presentation 
$$M= \{x\in\R^n: \mult \mx(x) \geq \rank \bar{v}\} \quad\textrm{ locally around } \bar{x},$$ where $\mult \mx(x)$ simply denotes the size of the set $I(x)$.
}
\end{exa}

Generalizing beyond polyhedrality, we now consider the so-called {\em piecewise linear-quadratic functions}; these are those functions whose domain can be represented as the union of finitely many convex polyhedra, so that the function is linear or quadratic on each such set. Convex piecewise linear-quadratic functions are precisely the convex functions whose subdifferential mappings are piecewise polyhedral \cite{piecequad}.  

\begin{prop}[Piecewise linear-quadratic functions]\label{prop:piece} {\ } \\ 
Consider a convex, piecewise linear-quadratic function $f\colon\R^n\to\overline{\R}$. Then there exists a locally minimal identifiable set at any point $x\in\dom f$ for any vector $y\in\partial f(x)$.  
\end{prop}
\begin{proof} 
Convex piecewise linear-quadratic functions have piecewise polyhedral subdifferential mappings \cite{piecequad}. Consequently, Example~\ref{exa:poly} shows that the mapping $x\mapsto \partial f(x)$ admits a locally minimal identifiable set at any point $x\in\R^n$ for any vector $v\in \partial f(x)$. Since piecewise linear-quadratic functions are lower-semicontinuous \cite[Proposition 10.21]{VA}, and lower-semicontinuous convex functions are subdifferentially continuous \cite[Example 13.30]{VA}, the result follows.
\end{proof}

We now briefly consider the three standard convex cones of mathematical programming.
\begin{exa}[Non-negative Orthant]\label{ex:pos_orth}
{\rm 
Consider a point $\bar{x}\in\R^n_{+}$ and a vector $\bar{v}\in N_{\R^n_{+}}(\bar{x})$. Then $M:=\{x\in\R^n_{+}: x_i=0 \textrm{ for each } i\in\supp \bar{v}\}$ is a locally minimal identifiable set at $\bar{x}$ for $\bar{v}$.
Observe that $M$ also admits the presentation  
$$M=\{x\in\R^n_{+}: \rank x +\rank \bar{v}\leq n\}\quad\textrm{ locally around }\bar{x}.$$
}
\end{exa}

\begin{exa}[Lorentz cone]\label{loren}
{\rm Consider the Lorentz cone $$\mathcal{L}^{n}:=\{(x,r)\in\R^n\times\R: r\geq |x|\}.$$ 
Observe that $\mathcal{L}^{n}$ coincides with the epigraph $\epi |\cdot|$. Let $\bar{x}=0$ and consider any $v\in\partial |\cdot|(0)$ with $|v|=1$.
Then for any real $\epsilon >0$, the set $M_{\epsilon}:=\{x\in\R^n: \langle \frac{x}{|x|},\bar{v} \rangle \leq \epsilon\}$ is identifiable at $\bar{x}$ for $\bar{v}$. In particular, for $n\geq 2$ and $\epsilon\neq \epsilon'$ the sets $M_{\epsilon}$ and $M_{\epsilon'}$ do not coincide on any neighborhood of $\bar{x}$, and consequently there is no locally minimal identifiable set at $\bar{x}$ for $\bar{v}$.
}
\end{exa}

In what follows ${\bf S}^n$ will denote the space of $n\times n$ real symmetric matrices with the trace inner product while ${\bf S}^n_{+}$ will denote the convex cone of symmetric positive semi-definite matrices. With every matrix $X\in {\bf S}^n$ we will associate its largest eigenvalue, denoted by $\lambda_1(X)$. The multiplicity of $\lambda_1(X)$ as an eigenvalue of $X$ will be written as $\mult\lambda_1(X)$. Finally ${\bf M}^{n\times m}$ will denote the space of $n\times m$ matrices with real entries.  
We defer the verification of the following two examples to a forthcoming paper \cite{spec_id}. We should also emphasize the intriguing parallel between these two examples and Examples~\ref{exa:max} and \ref{ex:pos_orth}.  
\begin{exa}[Positive semi-definite cone]\label{exa:sdp}
{\rm Consider a matrix $\bar{X}\in {\bf S}_{+}^n$ and a normal $\bar{V}\in N_{{\bf S}_{+}^n}(\bar{X})$.
Then $$M=\{X\in {\bf S}_{+}^n: \rank X+ \rank \bar{V} \leq n\},$$ is an identifiable set at $\bar{X}$
for $\bar{V}$. It is interesting to note that $M$ may fail to be locally minimal in general. Indeed, it is possible that ${\bf S}_{+}^n$ admits no locally minimal identifiable set at $\bar{X}$ for $\bar{V}$. This can easily be seen from the previous example and the fact that ${\bf S}_{+}^2$ and $\mathcal{L}^2$ are isometrically isomorphic.

However, under the strict complementarity condition $\bar{V}\in \ri N_{{\bf S}^n_{+}}(\bar{X})$, we have $\rank \bar{X}+\rank \bar{V}=n$, and consequently $M$ coincides with $\{X\in {\bf S}_{+}^n: \rank X =\rank \bar{X}\}$ around $\bar{X}$. It is then standard that $M$ is an analytic manifold around $\bar{X}$, and
furthermore one can show that $M$ is indeed a locally minimal identifiable set at $\bar{X}$ for $\bar{V}$. For more details see \cite{spec_id}.
} 
\end{exa} 

\begin{exa}[Maximum eigenvalue]
{\rm
Consider a matrix $\bar{X}$ and a subgradient $\bar{V}\in\partial\lambda_1 (\bar{X})$, where $\lambda_1\colon{\bf S}^n\to\R$ is the maximum eigenvalue function. Then $$M:=\{X\in {\bf S}^n: \mult \lambda_1(X) \geq \rank \bar{V}\},$$
is an identifiable set at $\bar{X}$ for $\bar{V}$. Again under a strict complementarity condition $\bar{V}\in\ri\partial\lambda_1 (\bar{X})$, we have $\rank \bar{V} = \mult \lambda_1(\bar{X})$, and consequently $M$ coincides with the manifold $\{X\in {\bf S}^n: \mult \lambda_1(X) =\mult \lambda_1(\bar{X})\}$ locally around $\bar{X}$. Furthermore under this strict complementarity condition, $M$ is locally minimal. For more details see \cite{spec_id}.}
\end{exa}

\begin{exa}[The rank function]
{\rm
Consider the rank function, denoted $\rank\colon {\bf M}^{n\times m}\to \R$. Then $$M:=\{X\in {\bf M}^{n\times m}:\rank X=\rank \bar{X}\}$$ is a locally minimal identifiable set at $\bar{X}$ for any $\bar{V}\in\partial(\rank)(\bar{X})$. To see this, observe that the equality
$$\epi \rank =\epi (\rank \bar{X}+\delta_M) \textrm{ holds locally around } (\bar{X},\rank \bar{X}).$$ Combining this with the standard fact that $M$ is an analytic manifold verifies the claim.
}
\end{exa}

In Examples~\ref{loren} and \ref{exa:sdp}, we already saw that there are simple functions $f\colon\R^n\to\overline{\R}$ that do not admit a locally minimal identifiable set at some point $\bar{x}$ for $\bar{v}\in\partial f(\bar{x})$. However in those examples $\bar{v}$ was degenerate in the sense that $\bar{v}$ was contained in the relative boundary of $\partial f(\bar{x})$. We end this section by demonstrating that locally minimal identifiable sets may, in general, fail to exist even for subgradients $\bar{v}$ lying in the relative interior of the convex subdifferential $\partial f(\bar{x})$.

\begin{exa}[Failure of existence]{\ }{\\}
{\rm
Consider the convex function $f\colon\R^2\to\R$, given by $$f(x,y)=\sqrt{x^4+y^2}.$$ 
Observe that $f$ is continuously differentiable on $\R^2\setminus \{(0,0)\}$, with  
$$|\nabla f(x,y)|^2 = \frac{4x^6+y^2}{x^4+y^2},$$
and $$\partial f(0,0)=\{0\}\times [-1,1].$$

We claim that $f$ does not admit a locally minimal identifiable set at $(0,0)$ for the vector $(0,0)\in \partial f(0,0)$. To see this, suppose otherwise and let $M$ be such a set. 

Consider the curves  $$L_{n}:=\{(x,y)\in\R^2:y=\frac{1}{n}x^2\},$$
parametrized by integers $n$. For a fixed integer $n$, consider a sequence of points $(x_i,y_i)\to(0,0)$ in $L_{n}$. Then 
$$\lim_{i\to\infty} |\nabla f(x_i,y_i)|=\frac{n^2}{n^4+1}.$$
Since $M$ is necessary at $(0,0)$ for $(0,0)$, we deduce that for each integer $n$, there exists a real number $\epsilon_n>0$ such that
$${\bf B}_{\epsilon_n}\cap L_{n}\cap M=\{(0,0)\}.$$

However observe $\lim_{n\to\infty} \frac{n^2}{n^4+1}=0$. Therefore we can choose a sequence $(x_n,y_n)\in {\bf B}_{\epsilon_n}\cap L_{n}$, with $(x_n,y_n)\neq (0,0)$, $(x_n,y_n)\to (0,0)$, and the gradients $\nabla f(x_n,y_n)$ tending to $(0,0)$. Since $M$ is identifiable at $(0,0)$ for $(0,0)$, the points $(x_n,y_n)$ lie in $M$ for all large indices $n$, which is a contradiction. 
}
\end{exa}

\section{Calculus of identifiability}\label{sec:calculus}
To build more sophisticated examples, it is necessary to develop some calculus rules. Our starting point is the following intuitive chain rule.

\begin{prop}[Chain Rule]\label{prop:chain} Consider a function $f(x):=g(F(x))$ defined on an open neighborhood $V\subset\R^n$, where $F\colon V\to\R^m$ is a ${\bf C}^1$-smooth mapping and $g\colon\R^n\to\overline{\R}$ is a lsc function. Suppose that at some point $\bar{x}\in\dom f$, the qualification condition 
\begin{equation}
\kernel \nabla F(\bar{x})^{*} \cap \partial^{\infty}g(F(\bar{x}))=\{0\}, \label{eq:chain_qual0}
\end{equation}
is valid, and hence the inclusion 
$$\partial f(\bar{x})\subset \nabla F(\bar{x})^{*}\partial g(F(\bar{x})) \textrm{ holds}.$$
Consider a vector $\bar{v}\in\partial f(\bar{x})$ and the corresponding multipliers $$\Lambda:=\{y\in\partial g(F(\bar{x})):\bar{v}=\nabla F(\bar{x})^{*}y\}.$$ Suppose that for each vector $y\in \Lambda$, there exists an identifiable set $M_y$ (with respect to $g$) at $F(\bar{x})$ for $y$. Then the set $$M:=\bigcup_{y\in \Lambda} F^{-1}(M_y),$$ is identifiable (with respect to $f$) at $\bar{x}$ for $\bar{v}$.

If, in addition, 
\begin{itemize}
\item $g$ is Clarke regular at all points in $\dom g$ around $F(\bar{x})$, 
\item the collection $\{M_y\}_{y\in S}$ is finite, and 
\item each set $M_y$ is a locally minimal identifiable set (with respect to $g$) at $F(\bar{x})$ for $y$, 
\end{itemize}
then $M$ is a locally minimal identifiable set (with respect to $f$) at $\bar{x}$ for $\bar{v}$.
\end{prop}
\begin{proof} 
We first argue the identifiability of $M$. To this effect, consider any sequence $(x_i,f(x_i),v_i)\to (\bar{x},f(\bar{x}),\bar{v})$, with $v_i\in\partial f(x_i)$. 
It is easy to see that the transversality condition 
\begin{equation}
\kernel \nabla F(x_i)^{*} \cap \partial^{\infty}g(F(x_i))=\{0\}, \label{eq:chain_qual2}
\end{equation}
holds for all sufficiently large indices $i$.
Then by \cite[Theorem 10.6]{VA}, we have 
$$v_i\in \partial f(x_i)\subset\nabla F(x_i)^{*}\partial g(F(x_i)).$$
Choose a sequence $y_i\in\partial g(F(x_i))$ satisfying $v_i=\nabla F(x_i)^{*}y_i$. We claim that the sequence $y_i$ is bounded. Indeed suppose otherwise. Then restricting to a subsequence, we can assume $|y_i|\to \infty$ and $\frac{y_i}{|y_i|}\to\tilde{y}$, for some nonzero vector $\tilde{y}\in\partial^{\infty}g(F(\bar{x}))$. Consequently 
$$\nabla F(\bar{x})^{*}\tilde{y}=\lim_{i\to\infty} \nabla F(x_i)^{*}\frac{y_i}{|y_i|}=\lim_{i\to\infty} \frac{v_i}{|y_i|}=0,$$ thus contradicting (\ref{eq:chain_qual2}).

Now restricting to a subsequence, we may suppose that the vectors $y_i\in\partial g(F(x_i))$ converge to $\bar{y}$ for some vector $\bar{y}\in\partial g(F(\bar{x}))$. Furthermore, observe $\bar{y}\in \Lambda$. So for all sufficiently large indices $i$, the points $F(x_i)$ all lie in $M_{\bar{y}}$. Consequently the points $x_i$ lie in $M$ for all large indices $i$, and we conclude that $M$ is identifiable (with respect to $f$) at $\bar{x}$ for $\bar{v}$.

Now suppose that $g$ is Clarke regular at all points of $\dom g$ near $F(\bar{x})$, the collection $\{M_y\}_{y\in S}$ is finite, and each set $M_y$ is a locally minimal identifiable set (with respect to $g$) at $F(\bar{x})$ for $y$. We now show that $M$ is necessary (with respect to $f$) at $\bar{x}$ for $\bar{v}$. To this effect, consider a sequence $x_i\to\bar{x}$ in $M$.
Then restricting to a subsequence, we may suppose that the points $F(x_i)$ all lie in $M_{\bar{y}}$ for some $\bar{y}\in \Lambda$. Consequently there exists a sequence $y_i\in\partial g(F(x_i))$ converging to $\bar{y}$. Hence we deduce 
$$v_i:=\nabla F(x_i)^{*}y_i\to \nabla F(\bar{x})^{*}\bar{y}=\bar{v}.$$ Since $g$ is Clarke regular at all points of $\dom g$ near $F(\bar{x})$, by \cite[Theorem 10.6]{VA}, the inclusion $v_i\in\partial f(x_i)$ holds for all large $i$. Hence $M$ is necessary (with respect to $f$) at $\bar{x}$ for $\bar{v}$.
\end{proof}

Our goal now is to obtain a sum rule. The passage to this result though the chain rule is fairly standard. The first step is to deal with separable functions.
\begin{prop}[Separable functions]\label{prop:sep_sum}{\ }{\\}
Consider proper, lsc functions $f_i\colon\R^{n_{i}}\to\overline{\R}$, for $i=1,\ldots,k$, and define $$f(x_1,\ldots,x_k)=\sum_{i=1}^k f_i(x_i).$$ 
Suppose that $M_{\bar{v}_i}\subset\R^{n_i}$ is an identifiable set (with respect to $f_i$) at $\bar{x}_i$ for $\bar{v}_i\in\partial f_i(\bar{x}_i)$, for each $i=1,\ldots,k$. Then the set $$M:=M_{\bar{v}_1}\times\ldots\times M_{\bar{v}_k},$$ is identifiable (with respect to $f$) at $\bar{x}=(\bar{x}_1,\ldots,\bar{x}_k)$ for $\bar{v}=(\bar{v}_1,\ldots,\bar{v}_k)$. An analogous result holds for necessary sets.
\end{prop}
\begin{proof}
Clearly $M:=M_{\bar{v}_1}\times\ldots\times M_{\bar{v}_k}$ is identifiable for the set-valued mapping  
$$(x_1,\ldots,x_k)\mapsto \prod_{i=1}^k \{f_i(x_i)\}\times \partial f_i(x_i),$$
at $\bar{x}$ for $\prod_{i=1}^k (f_i(\bar{x}_i),\bar{v}_i)$. Furthermore lower-semicontinuity of the functions $f_i$ readily implies that $M$ is also identifiable for 
$$(x_1,\ldots,x_k)\mapsto \{f(\bar{x})\}\times \prod_{i=1}^k \partial f_i(x_i),$$
at  $\bar{x}$ for $(f(\bar{x}),\bar{v})$. Using the identity $\partial f(x_1,\ldots,x_k)=\prod_{i=1}^k \partial f_i(x_i)$, we deduce the result. The argument in the context of necessary sets is similar.
\end{proof}

\begin{cor}[Sum Rule]\label{cor:sum}
Consider proper, lsc functions $f_i\colon\R^n\to\overline{\R}$, for $i=1,\ldots,k$, and define the sum $f(x)=\sum_{i=1}^k f_i(x)$. Assume  at some point $\bar{x}\in\dom f$, the qualification condition 
$$\sum_{i=1}^k v_i=0 \textrm{ and } v_i\in\partial f_i^{\infty}(\bar{x}) \textrm{ for each } i ~~\Longrightarrow~~ v_i=0 \textrm{ for each }i.$$
Consider a vector $\bar{v}\in\partial f(\bar{x})$ and define the set $$\Lambda=\{(v_1,\ldots, v_k)\in \prod_{i=1}^k \partial f_i(\bar{x}): \bar{v}=\sum_{i=1}^k v_i\}.$$ For each $(v_1,\ldots,v_k)\in \Lambda$, let $M_{v_i}$ be an identifiable set (with respect to $f_i$) at $\bar{x}$ for $v_i$. Then  $$M:=\bigcup_{(v_1,\ldots,v_k)\in \Lambda} M_{v_1}\cap\ldots\cap M_{v_k},$$ is identifiable (with respect to $f$) at $\bar{x}$ for $\bar{v}$.

If, in addition, 
\begin{itemize}
\item each $f_i$ is Clarke regular at all points in $\dom f_i$ around $\bar{x}$, 
\item the collection $\{M_{v_1}\times\ldots\times M_{v_k}\}_{(v_1,\ldots,v_k)\in \Lambda}$ is finite, and 
\item for each $(v_1,\ldots,v_k)\in \Lambda$, the set $M_{v_i}$ is a locally minimal identifiable set (with respect to $f_i$) at $\bar{x}$ for $v_i$, 
\end{itemize}
then $M$ is a locally minimal identifiable set (with respect to $f$) at $\bar{x}$ for $\bar{v}$.
\end{cor}
\begin{proof}
We may rewrite  $f$ in the composite form $g\circ F$, where $F(x):=(x,\ldots,x)$ and $g(x_1,\ldots,x_k):=\sum_{i=1}^k f_i(x_i)$ is separable. Then applying Proposition~\ref{prop:chain} and Proposition~\ref{prop:sep_sum} we obtain the result. 
\end{proof}

In particular, we now obtain the following geometric version of the chain rule. 
\begin{prop}[Sets with constraint structure]{\ }{\\}
Consider closed sets $Q\in\R^n$ and $K\in\R^m$, and a ${\bf C}^1$ smooth mapping $F\colon\R^n\to\R^m$. Define the set 
$$L=\{x\in Q: F(x)\in K\}.$$ Consider a pair $(\bar{x},\bar{v})\in\gph N_L$ and suppose that the constraint qualification 
\begin{equation*}
\left.
\begin{array}{c}
y\in N_Q(\bar{x}), w\in N_K(F(\bar{x})) \\
y+ \nabla F(\bar{x})^{*}w=0
\end{array}
\right\}
\Longrightarrow
(y,w)=(0,0),
\end{equation*}
holds. Define the set  
\begin{equation*}
\Lambda=\{(y,w)\in N_Q(\bar{x})\times N_K(F(\bar{x})): y+ \nabla F(\bar{x})^{*}w=\bar{v} \},
\end{equation*}
and for each pair $(y,w)\in \Lambda$, let $M_y$ be an identifiable set (relative to $Q$) at $\bar{x}$ for $y$ and let $K_w$ be an identifiable set (relative to $K$) at $F(\bar{x})$ for $w$. 
Then 
$$M:=\bigcup_{(v,w)\in \Lambda} M_v\cap F^{-1}(K_w),$$
is identifiable (relative to $L$) at $\bar{x}$ for $\bar{v}$.

If, in addition, 
\begin{itemize}
\item $Q$ (respectively $K$) is Clarke regular at each of its point near $\bar{x}$ (respectively $F(\bar{x})$), 
\item the collection $\{M_{y}\times K_{w}\}_{(y,w)\in \Lambda}$ is finite,
\item for each $(y,w)\in \Lambda$, the set $M_y$ (respectively $K_w$) is a locally minimal identifiable set with respect to $Q$ (respectively $K$) at $\bar{x}$ for $y$ (respectively at $F(\bar{x})$ for $w$),
\end{itemize}
then $M$ is a locally minimal identifiable set (relative to $L$) at $\bar{x}$ for $\bar{v}$.
\end{prop}
\begin{proof}
Observe $\delta_L=\delta_Q+\delta_{F^{-1}(K)}$. Combining Proposition~\ref{prop:chain} and Corollary~\ref{cor:sum}, we obtain the result.
\end{proof}

\begin{cor}[Max-type functions]{\ }{\\}
Consider ${\bf C}^1$-smooth functions $f_i\colon\R^n\to\overline{\R}$, for $i\in I:=\{1,\ldots,m\}$, and let $f(x):=\max\{f_1(x),f_2(x),\ldots,f_m(x)\}$.
For any $x\in\R^n$, define the active set 
$$I(x)=\{i\in I: f(x)=f_i(x)\}.$$
Consider a pair $(\bar{x},\bar{v})\in\gph \partial f$, and the corresponding set of multipliers
$$\Lambda=\{\lambda\in \R^m: \bar{v}=\sum_{i\in I(\bar{x})} \lambda_i \nabla f_i(\bar{x}), ~ \supp \lambda \subset I(\bar{x})\}.$$
Then $$M=\bigcup_{\lambda\in \Lambda} \{x\in\R^n: \supp \lambda\subset I(x)\},$$ is a locally minimal identifiable set (relative to $f$) at $\bar{x}$ for $\bar{v}$.
\end{cor}
\begin{proof}
This follows directly from Proposition~\ref{prop:chain} and Example~\ref{exa:conv_poly} by writing $f$ as the composition $\mx\circ F$, where $F(x)=(f_1(x),\ldots,f_m(x))$.
\end{proof}





\begin{cor}[Smooth constraints]{\ }{\\}
Consider ${\bf C}^1$-smooth functions $g_i\colon\R^n\to\overline{\R}$, for $i\in I:=\{1,\ldots,m\}$, and define the set
$$Q=\{x\in\R^n:g_i(x)\leq 0 \textrm{ for each } i\in I\}.$$
For any $x\in\R^n$, define the active set 
$$I(x)=\{i\in I: g_i(x)=0\}.$$
and suppose that for a certain pair $(\bar{x},\bar{v})\in\gph N_Q$, the constraint qualification

$$\sum_{i\in I(\bar{x})} \lambda_i\nabla g_i(\bar{x})=0 \textrm{ and } \lambda_i \geq 0 \textrm{ for all } i\in I(\bar{x}) ~~\Longrightarrow ~~ \lambda_i=0 \textrm{ for all }  i\in I(\bar{x}),$$
holds. Then in terms of the Lagrange multipliers 
$$\Lambda:=\{\lambda\in \R^m: \bar{v}=\sum_{i\in I(\bar{x})} \lambda_i \nabla g_i(\bar{x}), ~ \supp \lambda \subset I(\bar{x})\},$$
the set $$M=\bigcup_{\lambda\in \Lambda} \{x\in Q: g_j(x)= 0 \textrm{ for each }j\in \supp \lambda\},$$ is a locally minimal identifiable set (relative to $Q$) at $\bar{x}$ for $\bar{v}$.
\end{cor}
\begin{proof}
This follows immediately from Proposition~\ref{prop:chain} and Example~\ref{ex:pos_orth}.
\end{proof}

We end the section by observing that, in particular, the chain rule, established in Proposition~\ref{prop:chain}, allows us to consider the rich class of fully amenable functions, introduced in \cite{amen}. 
\begin{defn}[Fully amenable functions]{\ }{\\}
{\rm
A function $f\colon\R^n\to\overline{\R}$ is {\em fully amenable} at $\bar{x}$ if $f$ is finite at $\bar{x}$, and there is an open neighborhood $U$ of $\bar{x}$ on which $f$ can be represented as $f=g\circ F$ for a ${\bf C}^2$-smooth mapping $F\colon V\to\R^m$ and a convex, piecewise linear-quadratic function $g\colon\R^m\to\overline{\R}$, and such that the qualification condition 
$$\kernel \nabla F(\bar{x})^{*}\cap \partial g^{\infty}(F(\bar{x}))=\{0\},$$
holds.
}
\end{defn}

The qualification condition endows the class of fully amenable functions with exact calculus rules. Such functions are indispensable in nonsmooth second order theory. For more details, see \cite{amen}. 

\begin{prop}[Identifiable sets for fully amenable functions]\label{prop:amen}{\ }{\\}
A function $f\colon\R^n\to\overline{\R}$ that is fully amenable at a point $\bar{x}\in\R^n$ admits a locally minimal identifiable set at $\bar{x}$ for any vector $\bar{v}\in\partial f(\bar{x})$. 
\end{prop}
\begin{proof} This follows immediately from Propositions~\ref{prop:piece} and \ref{prop:chain}.
\end{proof}

\section{Variational geometry of identifiable sets}\label{sec:geo}
In the previous sections, we have introduced the notions of identifiability, analyzed when locally minimal identifiable sets exist, developed calculus rules, and provided important examples. In this section, we consider the interplay between variational geometry of a set $Q$ and its identifiable subsets $M$. Considering sets rather than functions has the advantage of making our arguments entirely geometric.
We begin with the simple observation that locally minimal identifiable sets are locally closed.
\begin{prop}\label{prop:closed}
Consider a closed set $Q\subset\R^n$ and a subset $M\subset Q$ that is a locally minimal identifiable set at $\bar{x}$ for $\bar{v}\in N_Q(\bar{x})$. Then $M$ is locally closed at $\bar{x}$
\end{prop}
\begin{proof}
Suppose not. Then there exists a sequence $x_i\in (\bd M)\setminus M$ with $x_i\to\bar{x}$. 
Since $M$ is identifiable at $\bar{x}$ for $\bar{v}$, there exists a neighborhood $V$ of $\bar{v}$ satisfying $\cl V\cap N_Q(x_i)=\emptyset$ for all large indices $i$. Observe that for each index $i$, every point $y$ sufficiently close to $x_i$ satisfies 
$V\cap N_Q(y)=\emptyset$. Consequently, there exists a sequence $y_i\in Q$ converging to $\bar{x}$ with $V\cap N_Q(y)=\emptyset$, which contradicts the necessity of $M$ at $\bar{x}$ for $\bar{v}$.  
\end{proof}

Recall that for a set $Q\subset\R^n$ and a subset $M\subset Q$, the inclusion $\hat{N}_Q(x)\subset N_M(x)$ holds for each point $x\in M$, while the analogous inclusion for the limiting normal cone may fail. This pathology does not occur for identifiable sets.

\begin{prop}\label{prop:subset}
Consider a closed set $Q$ and a set $M$ that is identifiable at $\bar{x}$ for $\bar{v}\in N_Q(\bar{x})$. Then the equation
$$\gph N_Q\subset \gph N_M~~ \textrm{ holds locally around } (\bar{x},\bar{v}).$$ 
\end{prop}
\begin{proof}
Consider a sequence $(x_i,v_i)\in\gph N_Q$ converging to $(\bar{x},\bar{v})$. Then for each $i$, there exists a sequence $(x^j_i,v^j_i)\in\gph \hat{N}_Q$ converging to $(x_i,v_i)$. For sufficiently large indices $i$, the points $x^j_i$ lie in $M$ for all large $j$. For such indices we have $v^j_i\in \hat{N}_Q(x^j_i)\subset\hat{N}_M(x^j_i)$, and consequently $v_i\in N_M(x_i)$. This verifies the inclusion 
$\gph N_Q\subset\gph N_M \textrm{ locally around } (\bar{x},\bar{v}).$
\end{proof}

It turns out that in order to make further headway in studying properties of identifiable sets, one must impose the condition of {\em prox-regularity}.
This concept has been discovered and rediscovered by various authors, notably by Federer~\cite{federer} and Rockafellar-Poliquin~\cite{prox_reg}. We follow the development of \cite{prox_reg} and \cite{theo}.
\begin{defn}[Prox-regularity]\label{defn:prox_glob}
{\rm
We say that a set $M\subset\R^n$ is {\em prox-regular} at $\bar{x}\in M$ if it is locally closed around $\bar{x}$ and there exists a neighborhood $U$ of $\bar{x}$, such that the projection map $P_M$ is single-valued on $U$.}
\end{defn}

Prox-regularity unifies the notions of convex sets and ${\bf C}^2$-manifolds. A proof may be found in {\cite[Example 13.30, Proposition 13.32]{VA}}.
\begin{thm}[Prox-regularity under convexity and smoothness]{\ }{\\}
Convex sets and ${\bf C}^2$-manifolds are prox-regular at each of their points.
\end{thm}



It will be useful to consider a variant of prox-regularity where we consider localization with respect to direction \cite[Definition 2.10]{prox_reg}.
\begin{defn}[Directional prox-regularity for sets]\label{defn:dir_prox}{\ }{\\}
{\rm Consider a set $Q\subset\R^n$, a point $\bar{x}\in Q$, and a normal vector $\bar{v}\in N_Q(\bar{x})$. We say that $Q$ is {\em prox-regular at} $\bar{x}$ {\em for} $\bar{v}$ if $Q$ is locally closed at $\bar{x}$ and there exist real numbers $\epsilon >0$ and $r >0$ such that the implication
\begin{equation*}
\left.
\begin{array}{cc}
x\in Q, &v\in N_Q(x) \\
|x-\bar{x}|<\epsilon, &|v-\bar{v}|<\epsilon
\end{array}
\right\}
\Rightarrow
P_{\mathsmaller{Q\cap B_{\epsilon}(\bar{x})}}(x+r^{-1}v)=x,
\end{equation*}
holds.
}
\end{defn}

Observe that if the implication above holds for some $r,\epsilon > 0$, then it also holds for any $r',\epsilon'>0$ with $r'>r$ and $\epsilon'<\epsilon$.

In particular, a set $Q$ is prox-regular at $\bar{x}$, in the sense of Definition~\ref{defn:prox_glob}, if and only of $Q$ is prox-regular at $\bar{x}$ for every vector $v\in N_Q(\bar{x})$. Clearly, if $Q$ is prox-regular at $\bar{x}$ for $\bar{v}$, then the equalities
$$\gph N^{P}_Q=\gph \hat{N}_Q=\gph N_Q\textrm{ hold locally around } (\bar{x},\bar{v}).$$

The following characterization \cite[Corollary 3.4]{prox_reg} will be of some use for us.
\begin{prop}[Prox-regularity and monotonicity]\label{prop:prox_char}{\ }{\\}
For a set $Q\subset\R^n$ and $\bar{x}\in Q$, with $Q$ locally closed at $\bar{x}$, the following are equivalent.
\begin{enumerate}
\item $Q$ is prox-regular at $\bar{x}$ for $\bar{v}$.
\item The vector $\bar{v}$ is a proximal normal to $Q$ at $\bar{x}$, and there exists a real number $r>0$ satisfying 
$$\langle v_1-v_0, x_1-x_0\rangle \geq -r|x_1-x_0|^2,$$ for any pairs $(x_i,v_i)\in\gph N_Q$  (for $i=0,1$) near $(\bar{x},\bar{v})$. 
\end{enumerate}
\end{prop}
So $Q$ is prox-regular at $\bar{x}$ for a proximal normal $\bar{v}\in N^{P}_Q(\bar{x})$ as long as $N_Q+rI$ has a monotone localization around $(\bar{x},\bar{v}+r\bar{x})$, for some real number $r>0$.
We may talk about prox-regularity of functions by means of epigraphical geometry \cite[Theorem 3.5]{prox_reg}.
\begin{defn}[Directional prox-regularity for functions]{\ }{\\}
{\rm We say that a function $f\colon\R^n\to\overline{\R}$ is {\em prox-regular at} $\bar{x}$ {\em for} $\bar{v}\in\partial f(\bar{x})$ if the epigraph $\epi f$ is prox-regular at $(\bar{x},f(\bar{x}))$ for the vector $(\bar{v},-1)$.
}
\end{defn}

The following proposition shows that prox-regularity of an identifiable subset $M$ of a set $Q$ implies that $Q$ itself is prox-regular. 
\begin{prop}[Prox-regularity of identifiable sets]\label{prop:prox}{\ }{\\}
Consider a closed set $Q$ and a subset $M\subset Q$ that is identifiable at $\bar{x}$ for $\bar{v}\in \hat{N}_Q(\bar{x})$. In addition, suppose that $M$ is prox-regular at $\bar{x}$ for $\bar{v}$. Then $Q$ is prox-regular at $\bar{x}$ for $\bar{v}$.
\end{prop}
\begin{proof}
To show that $Q$ is prox-regular at $\bar{x}$ for $\bar{v}$, we will utilize Proposition~\ref{prop:prox_char}. To this end, we first claim that the inclusion $\bar{v}\in N^{P}_Q(\bar{x})$ holds. To see this, choose a sequence of real numbers $r_i\to \infty$ and a sequence of points
$$x_i\in P_Q(\bar{x}+r_i^{-1} \bar{v}).$$ 
We have 
$$v_i:=r_i(\bar{x}-x_i)+\bar{v}\in N^{P}_Q(x_i).$$
Clearly $x_i\to \bar{x}$. We now claim that the sequence $v_i$ converges to $\bar{v}$. Observe by definition of $x_i$, we have 
$$|(\bar{x}- x_i)+r_i^{-1}\bar{v}|\leq |r_i^{-1}\bar{v}|.$$ Squaring and cancelling terms, we obtain 
$$2\langle\bar{v},\bar{x}-x_i \rangle\leq -r_i|\bar{x}-x_i|^2.$$  
Combining this with the inclusion $\bar{v}\in\hat{N}_Q(\bar{x})$, we deduce 
$$\frac{o(\bar{x}-x_i)}{|\bar{x}-x_i|}\leq 2\langle \bar{v},\frac{\bar{x}-x_i}{|\bar{x}-x_i|}\rangle \leq -r_i|\bar{x}-x_i|$$
We conclude $r_i(\bar{x}-x_i)\to 0$ and consequently $v_i\to\bar{v}$.
Since $M$ is identifiable at $\bar{x}$ for $\bar{v}$, we deduce $x_i\in M$ for all large indices $i$. In addition, since $M$ is prox-regular at $\bar{x}$ for $\bar{v}$, we have $$x_i=P_{M\cap B_{\epsilon}(\bar{x})}(\bar{x}+r_i^{-1} \bar{v})=\bar{x},$$ for some $\epsilon >0$ and for sufficiently large indices $i$. Hence the inclusion $\bar{v}\in N^{P}_Q(\bar{x})$ holds.

Now since $M$ is prox-regular at $\bar{x}$ for $\bar{v}$ we deduce, using Proposition~\ref{prop:prox_char}, that
there exists a real number $r>0$ satisfying 
$$\langle v_1-v_0, x_1-x_0\rangle \geq -r|x_1-x_0|^2,$$ for any pairs $(x_i,v_i)\in\gph N_M$  (for $i=0,1$) near $(\bar{x},\bar{v})$.

By Proposition~\ref{prop:subset}, we have 
$$\gph N_Q\subset \gph N_M ~\textrm{ locally around }(\bar{x},\bar{v}).$$
Recalling that $\bar{v}$ is a proximal normal to $Q$ at $\bar{x}$ and again appealing to Proposition~\ref{prop:prox_char}, we deduce that $Q$ is prox-regular at $\bar{x}$ for $\bar{v}$. 
\end{proof}

\begin{rem}
{\rm In Proposition~\ref{prop:prox}, we assumed that the vector $\bar{v}$ is a Frech\'{e}t normal. Without this assumption, the analogous result fails.
For instance, consider the sets $$Q:=\{(x,y)\in\R^2: xy=0\},$$ $$M:=\{(x,y)\in\R^2: y=0\}.$$ Then clearly $M$ is identifiable at $\bar{x}:=(0,0)$ for 
the normal vector $\bar{v}:=(0,1)\in N_Q(\bar{x})$. However, $\bar{v}$ is not a proximal normal. } 
\end{rem}

The following result brings to the fore the insight one obtains by combining the notions of identifiability and prox-regularity. It asserts that given a prox-regular identifiable set $M$ at $\bar{x}$ for $\bar{v}\in\hat{N}_Q(\bar{x})$, not only does the inclusion $\gph N_Q\subset \gph N_M$ hold locally around $(\bar{x},\bar{v})$, but rather the two sets $\gph N_Q$ and $\gph N_M$ coincide around $(\bar{x},\bar{v})$. 

\begin{prop}[Reduction I]\label{prop:even}
Consider a closed set $Q$ and let $M\subset Q$ be a set that is prox-regular at a point $\bar{x}$ for $\bar{v}\in\hat{N}_Q(\bar{x})$. Then $M$ is identifiable at $\bar{x}$ for $\bar{v}$ if and only if $$\gph N_Q=\gph N_M ~\textrm{ locally around } (\bar{x},\bar{v}).$$
\end{prop}
\begin{proof}
We must show that locally around $(\bar{x},\bar{v})$, we have the equivalence 
$$\gph N_Q\subset M\times\R^n \Leftrightarrow \gph N_Q=\gph N_M.$$
The implication ``$\Leftarrow$'' is clear. Now assume $\gph N_Q\subset M\times \R^n$ locally around $(\bar{x},\bar{v})$.
By prox-regularity, there exist real numbers $r,\epsilon >0$ so that $P_Q(\bar{x}+r^{-1}\bar{v})=\bar{x}$ (Proposition~\ref{prop:prox}) and so that the implication 
\begin{equation*}
\left.
\begin{array}{cc}
x\in M, &v\in N_M(x) \\
|x-\bar{x}|<\epsilon, &|v-\bar{v}|<\epsilon
\end{array}
\right\}
\Rightarrow
P_{\mathsmaller{M\cap B_{\epsilon}(\bar{x})}}(x+r^{-1}v)=x,
\end{equation*}
holds.
By Proposition~\ref{prop:subset}, it is sufficient to argue that the inclusion 
$$\gph N_M\subset \gph N_Q \textrm{ holds locally around } (\bar{x},\bar{v}).$$ 
Suppose this is not the case. Then there exists a sequence $(x_i,v_i)\to(\bar{x},\bar{v})$, with $(x_i,v_i)\in \gph N_M$ and $(x_i,v_i)\notin \gph N_Q$. 
Let $z_i\in P_Q(x_i+r^{-1}v_i)$. We have 

\begin{align}
(x_i-z_i)+&r^{-1}v_i\in N^{P}_Q(z_i),\notag \\
&x_i\neq z_i. \label{eqn:neq}
\end{align} 
Observe $z_i\to\bar{x}$ by the continuity of the projection map. Consequently, by the finite identification property, for large indices $i$, we have $z_i\in M$ and 
$$x_i+r^{-1}v_i \in z_i + N^{P}_Q(z_i)\subset z_i +N_M(z_i).$$
Hence $z_i=P_{\mathsmaller{M\cap B_{\epsilon}(\bar{x})}}(x_i+r^{-1}v_i)=x_i$, for large $i$, thus contradicting (\ref{eqn:neq}). 
\end{proof}

Recall that Proposition~\ref{prop:prox} shows that prox-regularity of an identifiable subset $M\subset Q$ is inherited by $Q$. It is then natural to consider to what extent the converse holds. It clearly cannot hold in full generality, since identifiable sets may contain many extraneous pieces. However we will see shortly that the converse does hold for a large class of identifiable sets $M$, and in particular for ones that are locally minimal. The key tool is the following lemma, which may be of independent interest.
\begin{lem}[Accessibility]\label{lem:acc}
Consider a closed set $Q\subset\R^n$ and a subset $M\subset Q$ containing a point $\bar{x}$. Suppose that for some vector $\bar{v}\in\hat{N}_Q(\bar{x})$, there exists a sequence 
$$y_i\in N^{P}_M(\bar{x})\setminus N^{P}_Q(\bar{x}) \textrm{ with } y_i\to\bar{v}.$$ 
Then there exists a sequence $(x_i,v_i)\in \gph N^{P}_Q$ converging to $(\bar{x},\bar{v})$ with $x_i\notin M$ for each index $i$.
\end{lem}
\begin{proof}
For each index $i$, there exists a real $r_i >0$ satisfying $P_M(\bar{x}+r_i^{-1}y_i)=\{\bar{x}\}$. Furthermore we can clearly assume $r_i\to \infty$. Define a sequence $(x_i,v_i)\in\gph N^P_Q$ by $$x_i\in P_Q(\bar{x}+r_i^{-1}y_i) ~~ \textrm{ and }~~ v_i:=r_i(\bar{x}-x_i)+y_i.$$
Observe $x_i\notin M$ since otherwise we would have $x_i=\bar{x}$ and $y_i=r_i v_i\in N^{P}_Q(\bar{x})$, a contradiction. By continuity of the projection $P_Q$, clearly we have $x_i\to\bar{x}$. Now observe 
$$|(\bar{x}-x_i)+r_i^{-1}y_i|\leq r_i^{-1}|y_i|.$$ Squaring and simplifying we obtain $$r_i|\bar{x}-x_i|+2\Big\langle \frac{\bar{x}-x_i}{|\bar{x}-x_i|},y_i\Big\rangle\leq 0.$$ Since $\bar{v}$ is a Frech\'{e}t normal, we deduce 
$$\lf_{i\to\infty}\, \Big\langle \frac{\bar{x}-x_i}{|\bar{x}-x_i|},y_i\Big\rangle= \lf_{i\to\infty}\, \Big\langle \frac{\bar{x}-x_i}{|\bar{x}-x_i|},\bar{v}\Big\rangle\geq 0.$$ Consequently we obtain $r_i|\bar{x}-x_i|\to 0$ and $v_i\to\bar{v}$, as claimed.
\end{proof}

\begin{prop}[Reduction II]\label{prop:red2}
Consider a closed set $Q\subset\R^n$, a point $\bar{x}$, and a normal $\bar{v}\in N_Q(\bar{x})$. Suppose $\gph N^P_Q =\gph N_Q \textrm{ locally around } (\bar{x},\bar{v})$, and consider a set $M:=N^{-1}_Q(V)$, where $V$ is a convex, open neighborhood of $\bar{v}$.
Then the equation 
$$\gph N_Q=\gph N_M \textrm{ holds locally around } (\bar{x},\bar{v}).$$
\end{prop}
\begin{proof}
First observe that since $M$ is identifiable at $\bar{x}$ for $\bar{v}$, applying Proposition~\ref{prop:subset}, we deduce that the inclusion $\gph N_Q\subset\gph N_M$ holds locally around $(\bar{x},\bar{v})$. To see the reverse inclusion, suppose that there exists a pair $(x,v)\in\gph N^P_M$, arbitrarily close to $(\bar{x},\bar{v})$, with $v\in V$ and $v\notin N^P_Q(x)$. By definition of $M$, we have $N^P_Q(x)\cap V\neq \emptyset$. Let $z$ be a vector in this intersection, and consider the line segment $\gamma$ joining $z$ and $v$. Clearly the inclusion $\gamma\subset V\cap N^P_M(x)$ holds. Observe that the line segment $\gamma\cap N^P_Q(x)$ is strictly contained in $\gamma$, with $z$ being one of its endpoints. Let $w$ be the other endpoint of $\gamma\cap N^P_Q(x)$. Then we immediately deduce that there exists a sequence $$y_i\in N^{P}_M(x)\setminus N^{P}_Q(x) \textrm{ with } y_i\to w.$$ Applying Lemma~\ref{lem:acc}, we obtain a contradiction. Therefore the inclusion $\gph N_Q\supset\gph N^P_M$ holds locally around $(\bar{x},\bar{v})$. Taking the closure the result follows.
\end{proof}

In particular, we obtain the following essential converse of Proposition~\ref{prop:prox}.
\begin{prop}[Prox-regularity under local minimality] {\ }{\\}
Consider a closed set $Q$ and a subset $M\subset Q$ that is a locally minimal identifiable set at $\bar{x}$ for $\bar{v}\in \hat{N}_Q(\bar{x})$. Then $Q$ is prox-regular at $\bar{x}$ for $\bar{v}$ if and only if $M$ is prox-regular at $\bar{x}$ for $\bar{v}$.
\end{prop}
\begin{proof}
The implication $\Leftarrow$ was proven in Proposition~\ref{prop:prox}. To see the reverse implication, first recall that $M$ is locally closed by Proposition~\ref{prop:closed}. Furthermore Propositions~\ref{prop:charsub} shows that there exists an open convex neighborhood $V$ of $\bar{v}$ so that $M$ coincides locally with $N^{-1}_Q(V)$. In turn, applying Proposition \ref{prop:red2} we deduce that the equation 
$$\gph N_Q=\gph N_M \textrm{ holds locally around } (\bar{x},\bar{v}).$$
Finally by Proposition~\ref{prop:prox_char}, prox-regularity of $Q$ at $\bar{x}$ for $\bar{v}$ immediately implies that $M$ is prox-regular at $\bar{x}$ for $\bar{v}$.
\end{proof}

We end this section by exploring the strong relationship between identifiable sets and the metric projection map. We begin with the following proposition.
\begin{prop}[Identifiability and the metric projection]{\ }{\\}
Consider a closed set $Q$ and a subset $M\subset Q$. Let $\bar{x}\in M$ and $\bar{v}\in N^P_Q(\bar{x})$. Consider the following conditions.
\begin{enumerate}
\item\label{it:usual2} $M$ is identifiable (relative to $Q$) at $\bar{x}$ for $\bar{v}$.
\item\label{it:proj2} For all sufficiently small $\lambda >0$, the set $M$ is identifiable (relative to $P_Q^{-1}$) at $\bar{x}$ for $\bar{x}+\lambda\bar{v}$.
\end{enumerate}  
Then the implication $\ref{it:usual2}\Rightarrow\ref{it:proj2}$ holds. If in addition $Q$ is prox-regular at $\bar{x}$ for $\bar{v}$, then the equivalence $\ref{it:usual2}\Leftrightarrow\ref{it:proj2}$ holds.
\end{prop}
\begin{proof}
$\ref{it:usual2}\Rightarrow\ref{it:proj2}$: Recall that for all small $\lambda >0$, we have $P_Q(\bar{x}+\lambda\bar{v})=\bar{x}$. Fix such a real number $\lambda$ and consider a sequence $(x_i,y_i)\to(\bar{x},\bar{x}+\lambda\bar{v})$ in $\gph P_Q^{-1}$. Observe $x_i\in P_Q(y_i)$ and the sequence $\lambda^{-1}(y_i-x_i)\in N_Q(x_i)$ converges to $\bar{v}$. Consequently, the points $x_i$ all eventually lie in $M$.

Suppose now that $Q$ is prox-regular at $\bar{x}$ for $\bar{v}$.

$\ref{it:proj2}\Rightarrow \ref{it:usual2}$: We may choose $\lambda,\epsilon >0$ as in Definition~\ref{defn:dir_prox}, and satisfying $P_Q(\bar{x}+\lambda\bar{v})=\bar{x}$. Consider a sequence $(x_i,v_i)\to(\bar{x},\bar{v})$ in $\gph N_Q$. Then the sequence $(x_i,x_i+\lambda v_i)$ converges to $(\bar{x},\bar{x}+\lambda \bar{v})$ and lies in $\gph P_Q^{-1}$ for all sufficiently large $i$. Consequently, we have $x_i\in M$ for all large $i$.
\end{proof}

Assuming prox-regularity, a simple way to generate identifiable subsets $M\subset Q$ is by projecting open sets onto $Q$. 
\begin{prop}[Projections of neighborhoods are identifiable]\label{prop:proj_inter} {\ }{\\}
Consider a set $Q\subset\R^n$ that is prox-regular at $\bar{x}$ for $\bar{v}\in N_Q(\bar{x})$. Then for all sufficiently small $\lambda >0$, 
if the inclusion $\bar{x}+\lambda \bar{v}\in\inter U$ holds for some set $U$, then $P_Q(U)$ is identifiable at $\bar{x}$ for $\bar{v}$.
\end{prop}
\begin{proof}
Suppose $(x_i,v_i)\to(\bar{x},\bar{v})$ in $\gph N_Q$. Then by prox-regularity for all sufficiently small $\lambda>0$, we have  
$P_{Q}(x_i+\lambda v_i)=x_i$ and $x_i+\lambda v_i \in U$ for all large $i$. We deduce $x_i\in P_Q(U)$ for all large $i$, as we needed to show.   
\end{proof}

In fact, we will see shortly that under the prox-regularity assumption, {\em all} locally minimal identifiable sets arise in this way.  
\begin{lem}\label{lem:inter}
Consider a closed set $Q\subset\R^n$ and a subset $M$ that is identifiable at $\bar{x}$ for $\bar{v}\in N^P_Q(\bar{x})$. Then for all sufficiently small $\lambda >0$ and all $\epsilon >0$, the inclusion
$$\bar{x}+\lambda \bar{v}\in \inter \{x+\lambda v: x\in M, v\in N_Q(x), |x-\bar{x}|<\epsilon, |v-\bar{v}|<\epsilon\},$$
holds.
\end{lem}
\begin{proof}
For all sufficiently small $\lambda>0$, we have $P_Q(\bar{x}+\lambda \bar{v})=\bar{x}$. Consider a sequence $z_i\to \bar{x}+\lambda \bar{v}$ and choose points
$$x_i\in P_Q(z_i).$$
Observe $x_i\to\bar{x}$ and the vectors $\lambda^{-1}(z_i-x_i)\in N_Q(x_i)$ converge to $\bar{v}$. Consequently, the points $x_i$ lie in $M$ for all large $i$, and hence $z_i=x_i+\lambda(\lambda^{-1}(z_i-x_i))$ lie in the desired set eventually.
\end{proof}

\begin{prop}[Representing locally minimal identifiable sets] \label{prop:projid}{\ }{\\}
Consider a set $Q\subset\R^n$ that is prox-regular at $\bar{x}$ for $\bar{v}\in N_Q(\bar{x})$ and let $M$ be a locally minimal identifiable set at $\bar{x}$ for $\bar{v}$. For $\lambda,\epsilon>0$, define $$U:=\{x+\lambda v: x\in M, v\in N_Q(x), |x-\bar{x}|<\epsilon, |v-\bar{v}|<\epsilon\}.$$ Then for all sufficiently small $\lambda,\epsilon>0$, we have $\bar{x}+\lambda\bar{v}\in\inter U$ and $M$ admits the presentation
$$M=P_Q(U) \textrm{ locally around }\bar{x}.$$
\end{prop}
\begin{proof}
Using prox-regularity and Lemma~\ref{lem:acc}, we deduce that for all sufficiently small $\lambda,\epsilon >0$ we have 
\begin{equation*}
\left.
\begin{array}{cc}
x\in Q, &v\in N_Q(x) \\
|x-\bar{x}|<\epsilon, &|v-\bar{v}|<\epsilon
\end{array}
\right\}
\Rightarrow
P_{\mathsmaller{Q}}(x+\lambda v)=x,
\end{equation*}
and $\bar{x}+\lambda\bar{v}\in\inter U$.
Using the fact that $M$ is locally minimal at $\bar{x}$ for $\bar{v}$, it is easy to verify that $M$ and $P_Q(U)$ coincide locally around $\bar{x}$. 
\end{proof}


Finally, we record the following characterization of identifiable sets.
\begin{prop}[Characterization of identifiable sets]{\ }{\\}
Consider a closed set $Q\subset\R^n$, a subset $M\subset Q$, a point $\bar{x}\in M$, and a normal vector $\bar{v}\in \hat{N}_Q(\bar{x})$. Consider the properties:
\begin{enumerate}
\item\label{it:proj_1} $M$ is identifiable at $\bar{x}$ for $\bar{v}$.
\item\label{it:proj_2} $Q$ is prox-regular at $\bar{x}$ for $\bar{v}$ and for all sufficiently small $\lambda>0$ and all $\epsilon >0$ the inclusion
$$\bar{x}+\lambda \bar{v}\in \inter \{x+\lambda v: x\in M, v\in N_Q(x), |x-\bar{x}|<\epsilon, |v-\bar{v}|<\epsilon\},$$
holds.
\end{enumerate}
Then the implication $(\ref{it:proj_2})\Rightarrow(\ref{it:proj_1})$ holds. If $M$ is prox-regular at $\bar{x}$ for $\bar{v}$,
then we have the equivalence $(\ref{it:proj_1})\Leftrightarrow(\ref{it:proj_2})$.
\end{prop}
\begin{proof}
$(\ref{it:proj_2})\Rightarrow(\ref{it:proj_1})$: 
By Proposition~\ref{prop:proj_inter}, for all sufficiently small $\lambda>0$ and all $\epsilon >0$, the set $$P_Q\Big(\{x+\lambda v: x\in M, v\in N_Q(x), |x-\bar{x}|<\epsilon, |v-\bar{v}|<\epsilon\}\Big),$$ is identifiable at $\bar{x}$ for $\bar{v}$. Furthermore, for all sufficiently small $\lambda,\epsilon>0$ this set is contained in $M$ locally around $\bar{x}$. Consequently $M$ is identifiable at $\bar{x}$ for $\bar{v}$.

Now suppose that $M$ is prox-regular at $\bar{x}$ for $\bar{v}$.

$(\ref{it:proj_1})\Rightarrow(\ref{it:proj_2}):$ This follows trivially from Proposition~\ref{prop:prox} and Lemma~\ref{lem:inter}.
\end{proof}

\section{Identifiable sets and critical cones}\label{sec:crit}
In this section, we consider {\em critical cones}, a notion that has been instrumental in sensitivity analysis, particularly in connection with polyhedral variational inequalities. See \cite[Section 2E]{imp} for example. We will see that there is a strong relationship between these objects and locally minimal identifiable sets. We begin with the notion of tangency.

\begin{defn}[Tangent cones]
{\rm Consider a set $Q\subset\R^n$ and a point $\bar{x}\in Q$. The {\em tangent cone} to $Q$ at $\bar{x}$, written $T_Q(\bar{x})$, consists of all vectors $w$ such that 
$$w=\lim_{i\to\infty} \frac{x_i-\bar{x}}{\tau_i},~ \textrm{ for some } x_i\stackrel{Q}{\rightarrow}\bar{x},~ \tau_i\downarrow 0.$$}
\end{defn}

The tangent cone is always closed but may easily fail to be convex. For any cone $K\in\R^n$, we consider the polar cone 
$$K^{*}:=\{y: \langle y,v\rangle\leq 0 \textrm{ for all } v\in K\}.$$ 
It turns out that the sets $\cl\convv T_Q(\bar{x})$ and $\hat{N}_Q(\bar{x})$ are dual to each other, that is the equation
$$\hat{N}_Q(\bar{x})=T_Q(\bar{x})^{*},$$
holds \cite[Theorem 6.28]{VA}. Consequently if $Q$ is locally closed at $\bar{x}$, then $Q$ is Clarke regular at $\bar{x}$ if and only if the equation $N_Q(\bar{x})=T_Q(\bar{x})^{*}$ holds.

A companion notion to tangency is smooth derivability.
\begin{defn}[smooth derivability]
{\rm
Consider a set $Q$ and a point $\bar{x}\in Q$. Then a tangent vector $w\in T_Q(\bar{x})$ is {\em smoothly derivable} if there exists a ${\bf C}^1$-smooth path $\gamma\colon [0,\epsilon)\to Q$ 
satisfying
$$w=\lim_{t\downarrow 0}\frac{\gamma(t)-\bar{x}}{t},$$
where $\epsilon >0$ is a real number and $\gamma(0)=\bar{x}$. We will say that $Q$ is {\em smoothly derivable} at $\bar{x}$ if every tangent vector $w\in T_Q(\bar{x})$ is smoothly derivable.}
\end{defn}

We should note that there is a related weaker notion of geometric derivability, where the path $\gamma$ is not required to be ${\bf C}^1$-smooth. For more details see \cite[Definition 6.1]{VA}.

Most sets that occur in practice are smoothly derivable. In particular, any smooth manifold is smoothly derivable at each of its point, as is any semi-algebraic set $Q\subset\R^n$. We omit the proof of the latter claim, since it is a straightforward consequence of the curve selection lemma \cite[Property 4.6]{DM} and the details needed for the argument would take us far off field. For a nice survey on semi-algebraic geometry, see \cite{Coste-semi}.

We now arrive at the following central notion. 
\begin{defn}[Critical cones]
{\rm
For a set $Q\subset\R^n$ that is Clarke regular at a point $\bar{x}\in Q$, the {\em critical cone} to $Q$ at $\bar{x}$ for $\bar{v}\in N_Q(\bar{x})$ is the set $$K_Q(\bar{x},\bar{v}):=N_{N_Q(\bar{x})}(\bar{v}).$$
}
\end{defn}

Because of the polarity relationship between normals and tangents, the critical cone $K_Q(\bar{x},\bar{v})$ can be equivalently described as
$$K_Q(\bar{x},\bar{v})=T_Q(\bar{x})\cap \bar{v}^{\perp},$$ where $\bar{v}^{\perp}$ is the subspace perpendicular to $\bar{v}$. For more information about critical cones and their use in variational inequalities and complementarity problems, see \cite{fac_pang}.

Connecting the classical theory of critical cones to our current work, we will now see that critical cones provide tangential approximations to locally minimal identifiable sets.
In what follows, we denote the closed convex hull of any set $Q\subset\R^n$ by $\clco Q$.

\begin{prop}[Critical cones as tangential approximations]\label{prop:tang}{\ }{\\}
Consider a set $Q$ that is Clarke regular at a point $\bar{x}$ and a  locally minimal identifiable set $M$ at $\bar{x}$ for $\bar{v}\in N_Q(\bar{x})$. Suppose furthermore that $M$ is prox-regular at $\bar{x}$ for $\bar{v}$ and is smoothly derivable at $\bar{x}$. Then the equation 
$$\clco T_M(\bar{x})= K_Q(\bar{x},\bar{v}),$$
holds.
\end{prop}
\begin{proof}
Observe $$K_Q(\bar{x},\bar{v})=N_{N_Q(\bar{x})}(\bar{v})=N_{N_M(\bar{x})}(\bar{v})= N_{\hat{N}_M{(\bar{x})}}(\bar{v})=\clco T_M(\bar{x})\cap \bar{v}^{\perp},$$
where the second equality follows from Proposition~\ref{prop:even} and the last equality follows from polarity of $\cl\convv T_M(\bar{x})$ and $\hat{N}_M(\bar{x})$. Hence to establish the claim, it is sufficient to argue that every tangent vector $w\in T_M(\bar{x})$ is orthogonal to $\bar{v}$.

To this end, fix a vector $w\in T_M(\bar{x})$ and a ${\bf C}^1$-smooth path $\gamma\colon [0,\epsilon)\to Q$ 
satisfying
$$w=\lim_{t\downarrow 0}\frac{\gamma(t)-\bar{x}}{t},$$
where $\epsilon >0$ is a real number and $\gamma(0)=\bar{x}$.

Let $t_i\in (0,\epsilon)$ be a sequence converging to $0$ and define $x_i:=\gamma(t_i)$. Observe that for each index $i$, the tangent cone $T_M(x_i)$ contains the line $\{\lambda \dot{\gamma}(t_i):\lambda\in\R\}$. 
Since $M$ is necessary at $\bar{x}$ for $\bar{v}$, there exist vectors $v_i\in N_Q(\gamma(t_i))$ with $v_i\to\bar{v}$. By Proposition~\ref{prop:even}, we have $v_i\in \hat{N}_M(\gamma(t_i))$ for all large $i$. For such indices, we have
$\langle v_i,\dot{\gamma}(t_i)\rangle =0$. Letting $i$ tend to $\infty$, we deduce $\langle \bar{v},\bar{w}\rangle =0$, as we needed to show. 
\end{proof}

Classically, the main use of critical cones has been in studying polyhedral variational inequalities. Their usefulness in that regard is due to Proposition~\ref{prop:poly_equiv}, stated below. We provide a simple proof of this proposition that makes it evident that this result is simply a special case of Proposition~\ref{prop:even}. This further reinforces the theory developed in our current work. For an earlier proof that utilizes representations of polyhedral sets, see for example \cite[Lemma 2E.4]{imp}.
\begin{prop}[Polyhedral reduction]\label{prop:poly_equiv}
Consider a polyhedron $Q\subset\R^n$ and a normal vector $\bar{v}\in N_Q(\bar{x})$, for some point $\bar{x}\in Q$. Let $K:=K_Q(\bar{x},\bar{v})$. Then we
have $$\gph N_Q - (\bar{x},\bar{v})= \gph N_K \textrm{ locally around } (0,0).$$
\end{prop}
\begin{proof}
By Example~\ref{exa:conv_poly}, the set $M:=\argmax_{x\in Q}\langle x,\bar{v}\rangle$ is the locally minimal identifiable set at $\bar{x}$ for $\bar{v}$.  Being polyhedral, $M$ is smoothly derivable and it satisfies $$\bar{x}+T_M(\bar{x})=M \textrm{ locally around } \bar{x}.$$ In light of Proposition~\ref{prop:tang}, we deduce $M-\bar{x}= K$ locally around $0$.

Thus for all $(u,w)$ sufficiently near $(0,0)$ we have
\begin{align*}
\bar{v}+u\in N_Q(\bar{x}+w) &\Longleftrightarrow\bar{v}+u\in N_M(\bar{x}+w)\\
&\Longleftrightarrow \bar{v}+u\in N_K(w)\\
&\Longleftrightarrow u\in N_K(w)
\end{align*}
where the first equivalence follows from Proposition~\ref{prop:even}, and
the last equivalence follows from the fact that $K\subset \bar{v}^{\perp}$ and so for all $w\in K$, the cone $N_K(w)$ contains the line spanned by $\bar{v}$.
\end{proof}

Proposition~\ref{prop:poly_equiv} easily fails for nonpolyhedral sets. Indeed, in light of Proposition~\ref{prop:tang}, this is to be expected since critical cones provide only tangential approximations to locally minimal identifiable sets. Such an approximation is exact only for polyhedral sets. Hence the theory of locally minimal identifiable sets (in particular, Proposition~\ref{prop:even}) extends Proposition~\ref{prop:poly_equiv} far beyond polyhedrality. 

We end this section by showing how Proposition~\ref{prop:tang} can be extended even further to the situation when locally minimal identifiable sets do not even exist. Indeed, consider a set $Q$ that is Clarke regular at a point $\bar{x}$, and let $\bar{v}\in N_Q(\bar{x})$. Consider a nested sequence of open neighborhoods $V_i$ of $\bar{v}$ satisfying $\bigcap_{i=1}^{\infty} V_i=\{\bar{v}\}$. One would then expect that, under reasonable conditions, the equality $$K_Q(\bar{x},\bar{v})=\clco \bigcap_{i=1}^{\infty} T_{N^{-1}_Q(V_i)}(\bar{x}),$$ holds. To put this in perspective, observe that if there exists a locally minimal identifiable set $M$ at $\bar{x}$ for $\bar{v}$, then the sets $T_{N^{-1}_Q(V_i)}(\bar{x})$ are equal to $T_M(\bar{x})$ for all large $i$, and  
the equation above reduces to Proposition~\ref{prop:tang}. More generally, the following is true.

\begin{prop}[Critical cones more generally]
Consider a set $Q$ that is Clarke regular at a point $\bar{x}$, and let $\bar{v}\in N_Q(\bar{x})$. Consider a nested sequence of open neighborhoods $V_i$ of $\bar{v}$ satisfying $\bigcap_{i=1}^{\infty} V_i=\{\bar{v}\}$ and the corresponding preimages $M_i:=\hat{N}^{-1}_Q(V_i)$. Assume that each $M_i$ is smoothly derivable at $\bar{x}$. Then the inclusion 
\begin{equation}\label{eq:inc_tang}
K_Q(\bar{x},\bar{v})\supset \clco \bigcap_{i=1}^{\infty} T_{M_i}(\bar{x}),
\end{equation}
holds.
Assume in addition that each $M_i$ is prox-regular at $\bar{x}$ for $\bar{v}$ and that the formula 
\begin{equation}\label{eq:commute}
\clco\bigcap_{i=1}^{\infty} T_{M_i}(\bar{x})=\bigcap_{i=1}^{\infty} \clco T_{M_i}(\bar{x}),
\end{equation}
holds. Then each $M_i$ is an identifiable set at $\bar{x}$ for $\bar{v}$ and we have 
$$K_Q(\bar{x},\bar{v})=\clco\bigcap_{i=1}^{\infty} T_{M_i}(\bar{x}).$$
\end{prop}

We omit the proof of the proposition above since it follows along the same lines as the proof of Proposition~\ref{prop:tang}. In particular, let us note that (\ref{eq:commute}) holds whenever the tangent spaces $T_{M_{i}}(\bar{x})$ all coincide for sufficiently large indices $i$ or whenever all $M_i$ are Clarke regular at $\bar{x}$.





\section{Optimality conditions}\label{sec:opt}
In this section, we will see that the order of growth of a function $f$ around a critical point (a point satisfying $0\in\partial f(x)$) is dictated entirely by its order of growth around this point on a corresponding identifiable set. Here is a preliminary geometric result.



\begin{prop}[Restricted optimality]\label{prop:rest:opt}
Consider a closed set $Q$ and a subset $M\subset Q$ that is identifiable at $\bar{x}$ for a proximal normal $\bar{v}\in N^P_Q(\bar{x})$. Then $\bar{x}$ is a (strict) local maximizer of the linear function $\langle \bar{v},\cdot \rangle$ on $M$ if and only if $\bar{x}$ is a (strict) local maximizer of $\langle \bar{v},\cdot \rangle$ on $Q$.

\end{prop}
\begin{proof}
One implication is clear. To establish the converse, suppose that $\bar{x}$ is a local maximizer of the linear function $\langle \bar{v},\cdot \rangle$ on $M$. We will show that the inequality, $\langle\bar{v},\bar{x} \rangle > \langle\bar{v},x \rangle$, holds for all points $x\in Q\setminus M$ near $\bar{x}$.
Indeed, suppose this is not the case. Then there exists a sequence $x_i\to\bar{x}$ in $Q\setminus M$ satisfying
\begin{equation}\label{eq:imp2}
\langle\bar{v},\bar{x} \rangle \leq \langle\bar{v},x_i \rangle.
\end{equation}
Since $\bar{v}$ is a proximal normal, we deduce that there exists a real number $r>0$ satisfying $P_Q(\bar{x}+r^{-1}\bar{v})=\{\bar{x}\}$.
Consider any points $z_i$ with $$z_i\in P_Q(x_i+r^{-1}\bar{v}).$$ Clearly we have $z_i\to\bar{x}$ and
$$(x_i-z_i)+r^{-1}\bar{v}\in N_Q(z_i).$$ Since $M$ is identifiable at $\bar{x}$ for $\bar{v}$, we deduce $z_i\in M$ for all large indices $i$. Consequently, for such indices $i$, we have $x_i\neq z_i$.

Observe $$|(x_i-z_i)+r^{-1}\bar{v}|\leq r^{-1}|\bar{v}|.$$ Squaring and canceling terms, we obtain
\begin{equation}\label{eq:imp3}
\langle\bar{v},z_i-x_i \rangle\geq \frac{r}{2}|z_i-x_i|^2.
\end{equation}
Consequently,
$$\frac{r}{2}|z_i-x_i|^2 \leq \langle \bar{v},\bar{x}-x_i\rangle\leq 0,$$
which is a contradiction. Claim (1) now follows.
\end{proof}

Recall that a function $f\colon\R^n\to\overline{\R}$ is said to {\em grow quadratically} around $\bar{x}$ provided that the inequality 
$$\lf_{x\to \bar{x}} \,\frac{f(x)-f(\bar{x})}{|x-\bar{x}|^2}>0,$$
holds. We now arrive at the main result of this section.
\begin{prop}[Order of growth]\label{prop:quad}
Consider a function $f\colon\R^n\to\overline{\R}$ and a set $M\subset\R^n$. Suppose that $M$ is identifiable at $\bar{x}$ for  $\bar{v}=0\in \partial_P f(\bar{x})$. Then the following are true.
\begin{enumerate}
\item $\bar{x}$ is a (strict) local minimizer of $f$ restricted to $M$ $\Leftrightarrow$ $\bar{x}$ is a (strict) local minimizer of the unrestricted function $f$.
\item More generally, consider a growth function $g\colon U\to\R$, defined on an open neighborhood $U$ of $0$, that is ${\bf C}^1$-smooth and satisfies
\begin{align*}
& f(\bar{x}) < f(x) - g(x-\bar{x}) \textrm{   for all } x\in M \textrm{ near } \bar{x},\\
&g(0)=0,~ \nabla g(0)=0, 
\end{align*}
Then the above inequality, in fact, holds for all points $x\in \R^n$ near $\bar{x}$.

In particular, the function $f$, restricted to $M$, grows quadratically near $\bar{x}$ if and only if the unrestricted function $f$ grows quadratically near $\bar{x}$.
\end{enumerate}
\end{prop}
\begin{proof} We first prove claim (1). By Proposition~\ref{prop:coherence}, $\gph f\big|_M$ is identifiable, with respect to $\epi f$, at $(\bar{x},f(\bar{x}))$ for $(0,-1)$. Now observe that $\bar{x}$ is a (strict) local minimizer of $f\big|_M$ if and only of $(\bar{x},f(\bar{x}))$ is a (strict) local maximizer of the linear function, $(x,r)\mapsto -r$, on $\gph f\big|_M$. Similarly $\bar{x}$ is a (strict) local minimizer of $f$ if and only of $(\bar{x},f(\bar{x}))$ is a (strict) local maximizer of the linear function, $(x,r)\mapsto -r $, on $\epi f$. Combining these equivalences with Proposition~\ref{prop:rest:opt} establishes the claim.

We now prove claim (2). Suppose that the growth condition is satisfied. Let $h:=f-g(x-\bar{x})$. Since $f$ is ${\bf C}^1$-smooth, $g(0)=0$, and $\nabla g(0)=0$, it easily follows that $M$ is identifiable, now with respect to $h$, at $\bar{x}$ for $0\in\partial_P h(\bar{x})$. Furthermore, the point $\bar{x}$ is a strict local minimizer of $h\big|_M$. Applying claim (1) of the current proposition, we deduce that $\bar{x}$ is a strict local minimizer of the unrestricted function $h$, that is
$$f(\bar{x})=h(\bar{x})<h(x)= f(x)-g(x-\bar{x}), \textrm{   for all } x\textrm{ near } \bar{x},$$ as we needed to show.
\end{proof}

In particular, we obtain the following curious characterization of quadratic growth.
\begin{cor}[Refined optimality]
Consider a function $f\colon\R^n\to\overline{\R}$ and a point $\bar{x}$ with $0\in\partial_P f(\bar{x})$. Then $f$ grows quadratically around $\bar{x}$ if and only if
\begin{equation}\label{eq:grow}
\lf_{\substack{(x,f(x),v)\to(\bar{x},f(\bar{x}),0)\\v\in\partial f(x)}} \frac{f(x)-f(\bar{x})}{|x-\bar{x}|^2}>0.
\end{equation}
\end{cor}
\begin{proof}
Clearly if $f$ grows quadratically around $\bar{x}$, then (\ref{eq:grow}) holds. Conversely, assume $(\ref{eq:grow})$ holds and let $V_i$ be a sequence of neighborhoods of $0$ shrinking to $0$ and let $\epsilon_i>0$ be real number tending to $0$. Then the sets $$M_i:=(\partial f)^{-1}(V_i)\cap \{x\in\R^n: |f(x)-f(\bar{x})|<\epsilon_i\},$$ are identifiable at $\bar{x}$ for $0$. Furthermore, $f$ restricted to $M_i$ must grow quadratically around $\bar{x}$, for all sufficiently large indices $i$, since the alternative would contradict (\ref{eq:grow}). Applying Proposition~\ref{prop:quad}, we obtain the result.  
\end{proof}

\section{Identifiable Manifolds} \label{sec:man}
Consider a closed set $Q$ and a normal vector $\bar{v}\in \hat{N}_Q(\bar{x})$, for a point $\bar{x}\in Q$. The inherent difficulty in analyzing properties of the optimization problem, 
\begin{align*}
P(v):~~~&\max ~~\langle  v,x\rangle,\\
&~~{\rm s.t.} ~~x\in Q,
\end{align*}
such as dependence of the local maximizers of $P(v)$ on $v$ or the order of growth of the function $x\mapsto \langle x,\bar{v}\rangle$ on $Q$ near $\bar{x}$, stem entirely from the potential nonsmoothness of $Q$. However, as we have seen in Proposition~\ref{prop:even}, the local geometry of $\gph N_Q$ is entirely the same as that of a prox-regular identifiable set $M$ at $\bar{x}$ for $\bar{v}$. Thus, for instance, existence of an {\em identifiable manifold}  $M$ at $\bar{x}$ for $\bar{v}$ shows that the nonsmoothness of $Q$ is not intrinsic to the problem at all. Our goal in this section is to investigate this setting. We begin with the following easy consequence of Proposition~\ref{prop:even}. 

\begin{prop}\label{prop:inner_set}
Consider a closed set $Q\subset\R^n$ and suppose that a subset $M\subset Q$ is a ${\bf C}^2$ identifiable manifold at $\bar{x}$ for $\bar{v}\in \hat{N}_Q(\bar{x})$. Then the following properties hold.
\begin{enumerate}
\item $\bar{v}$ lies in the interior of the cone $N^{P}_Q(\bar{x})$, relative to its linear span $N_M(\bar{x})$.
\item There exists an open neighborhood $U$ of $\bar{x}$ and $V$ of $\bar{v}$ such that the mapping $x\mapsto V \cap N_Q(x)$, restricted to $M$, is inner-semicontinuous at each $x\in U\cap M$.
\end{enumerate}
\end{prop}
\begin{proof}
To see the validity of the first claim, observe that if it did not hold, then we could choose a sequence of vectors $v_i$ satisfying
$$v_i\to \bar{v},~~ v_i\in N_M(\bar{x}),~~ v_i\notin N_Q(\bar{x}),$$ thus contradicting Proposition~\ref{prop:even}. 

The second claim now easily follows from Proposition~\ref{prop:even}.
\end{proof}

Consider a locally minimal identifiable subset $M\subset Q$ at $\bar{x}$ for $\bar{v}\in N_Q(\bar{x})$. Then $M$ remains identifiable at $x$ for $v\in N_Q(x)$, whenever the pair $(x,v)$ is sufficiently close to $(\bar{x},\bar{v})$. However under such perturbations, $M$ might cease to be locally minimal, as one can see even from polyhedral examples. (Indeed when $Q$ is a convex polyhedron, this instability occurs whenever the inclusion $\bar{v}\in\rb N_Q(\bar{x})$ holds.) 

In the case of identifiable manifolds, the situation simplifies. Identifiable manifolds at $\bar{x}$ for $\bar{v}\in \hat{N}_Q(\bar{x})$ are automatically locally minimal, and furthermore they remain locally minimal under small perturbations to $(\bar{x},\bar{v})$ in $\gph N_Q$.

This important observation is summarized below.
\begin{prop}\label{prop:aut}
Consider a closed set $Q$ and a ${\bf C}^2$ identifiable manifold $M\subset Q$ at $\bar{x}$ for $\bar{v}\in\hat{N}_Q(\bar{x})$. Then $M$ is automatically a locally minimal identifiable set at $x\in M$ for $v\in N_Q(x)$ whenever the pair $(x,v)$ is near $(\bar{x},\bar{v})$. 
\end{prop}
\begin{proof}
This follows directly from Proposition~\ref{prop:loc_min} and Proposition~\ref{prop:inner_set}.
\end{proof}

In particular, identifiable manifolds at $\bar{x}$ for $\bar{v}\in\partial f(\bar{x})$ are locally unique.

\subsection{Relation to Partial Smoothness}\label{sec:char}
In this section, we will relate identifiable manifolds to the notion of partial smoothness, introduced in \cite{Lewis-active}.
The motivation behind partial smoothness is two-fold. On one hand, it is an attempt to model an intuitive idea of a ``stable active set''. On the other hand, partial smoothness, along with certain nondegeneracy and growth conditions, provides checkable sufficient conditions for optimization problems to possess good sensitivity properties. 
Evidently, partial smoothness imposes conditions that are unnecessarily strong. We now describe a variant of partial smoothness that is localized in a directional sense. This subtle distinction, however, will be important for us.

\begin{defn}[Directional Partial Smoothness]
{\rm 
Consider a closed set $Q\subset\R^n$ and a ${\bf C}^2$-manifold $M\subset Q$. Then $Q$
is {\em partly smooth} {\em with respect
to} $M$ {\em at} $\bar{x}\in M$ {\em for} $\bar{v}\in N_Q(\bar{x})$ if  
\begin{enumerate}
\item {\bf (prox-regularity)}  $Q$ is prox-regular at $\bar{x}$ for $\bar{v}$.
\item {\bf (sharpness)} $\spann \hat{N}_Q(\bar{x})=N_M(\bar{x})$.  
\item {\bf (continuity)} There exists a neighborhood $V$ of $\bar{v}$, such that the mapping, $x\mapsto V\cap N_Q(x)$, when restricted 
to $M$, is inner-semicontinuous at $\bar{x}$.
\end{enumerate}
}
\end{defn}

We arrive at the main result of this subsection.
\begin{prop}[Identifiable manifolds and partial smoothness]\label{prop:eqv2}{\ }{\\}
Consider a closed set $Q\subset\R^n$ and a subset $M\subset Q$ that is a ${\bf C}^2$ manifold around a point $\bar{x}\in Q$. Let $\bar{v}\in \hat{N}_Q(\bar{x})$. Then the following are equivalent.
\begin{enumerate}
\item \label{item:ident} $M$ is an identifiable manifold at $\bar{x}$ for $\bar{v}$. 
\item \label{item:graph} We have 
$$\gph N_Q=\gph N_M ~\textrm{ locally around } (\bar{x},\bar{v}).$$
\item \label{item:part}
\begin{itemize}
\item $Q$ is partly smooth with respect to $M$ at $\bar{x}$ for $\bar{v}$. 
\item the strong inclusion $\bar{v}\in \ri \hat{N}_Q(\bar{x})$ holds. 
\end{itemize}
\item \label{item:val} The set $Q$ is prox-regular at $\bar{x}$ for $\bar{v}$, and for all sufficiently small real numbers $\lambda, \epsilon >0$, the inclusion
$$\bar{x}+\lambda \bar{v}\in \inter \Big(\bigcup_{x\in M\cap B_{\epsilon}(\bar{x})} \big(x+N_Q(x)\big)\Big),$$
holds.
\end{enumerate}
\end{prop}
\begin{proof} 
The equivalence $(\ref{item:ident})  \Leftrightarrow(\ref{item:graph})$ has been established in Proposition~\ref{prop:even}.
The implication $(\ref{item:ident})\Rightarrow (\ref{item:part})$ follows trivially from Propositions~\ref{prop:prox} and \ref{prop:inner_set}. 

$(\ref{item:part})\Rightarrow (\ref{item:val})$: There exist real numbers $r,\epsilon>0$ so that the implication 
\begin{equation*}
\left.
\begin{array}{cc}
x\in Q, &v\in N_Q(x) \\
|x-\bar{x}|<\epsilon, &|v-\bar{v}|<\epsilon
\end{array}
\right\}
\Rightarrow
P_{\mathsmaller{Q\cap B_{\epsilon}(\bar{x})}}(x+r^{-1}v)=x,
\end{equation*}
holds.

For the sake of contradiction, suppose $$\bar{x}+r^{-1}\bar{v}\in \bd \Big(\bigcup_{x\in M\cap B_{\epsilon}(\bar{x})} \big(x+N_Q(x)\big)\Big).$$
Then there exists a sequence of points $z_i\to \bar{x}+r^{-1}\bar{v}$ with $$z_i\notin \bigcup_{x\in M\cap B_{\epsilon}(\bar{x})} \big( x+N_Q(x)\big),$$ for each index $i$. Let 
$x_i\in P_M(z_i)$. Observe $$x_i\to \bar{x}, ~~z_i-x_i\in N_M(x_i).$$ Clearly, $$z_i-x_i\notin N^{P}_Q(x_i)\subset N_M(x_i),$$ for large indices $i$. Hence, there exist separating vectors $a_i\in N_M(x_i)$ with $|a_i|=1$ satisfying
$$\sup_{v\in N^P_Q(x_i)}\langle a_i, v\rangle \leq \langle a_i, z_i-x_i \rangle=\langle a_i,z_i-\bar{x}\rangle +\langle a_i,\bar{x}-x_i\rangle.$$ 
We deduce, 
$$\sup_{v\in N^P_Q(x_i)}\langle a_i, v\rangle \leq \langle a_i,r(z_i-\bar{x})\rangle +\langle a_i,r(\bar{x}-x_i)\rangle.$$ 
Passing to a subsequence, we may assume $a_i\to a$ for some nonzero vector $a\in N_M(\bar{x})$. Observe $\bar{v}+\delta a\in N_Q(x_i)$ for all small $\delta >0$. Consequently for all sufficiently small $\delta >0$, there exist vectors $v_i\in N_Q(x_i)$ with $v_i\to\bar{v}+\delta a$. Observe
$$\langle a_i,v_i\rangle\leq \langle a_i, r(z_i-\bar{x})\rangle+\langle a_i, r(\bar{x}-x_i)\rangle.$$ Letting $i$ tend to $\infty$, we obtain 
$$\langle a,\bar{v}+\delta a\rangle \leq\langle a,\bar{v}\rangle,$$ which is a contradiction.
 
$(\ref{item:val})\Rightarrow (\ref{item:ident})$: Choose $r,\epsilon >0$ so as to ensure $$\bar{x}+r^{-1}\bar{v}\in \inter \Big(\bigcup_{x\in M\cap B_{\epsilon}(\bar{x})} \big( x+N_Q(x)\big)\Big).$$ Consider any sequence of points $x_i\in\R^n$ and vectors $v_i\in N_Q(x_i)$, with $x_i\to\bar{x}$ and $v_i\to\bar{v}$. Then for all large indices $i$, the inclusion $$x_i+r^{-1}v_i\in \bigcup_{x\in M\cap B_{\epsilon}(\bar{x})} \big( x+N_Q(x)\big),$$
holds. Shrinking $r$ and $\epsilon$, from prox-regularity of $Q$, we deduce $x_i\in M$ for all large indices $i$. Hence $M$ is identifiable at $\bar{x}$ for $\bar{v}$. 
\end{proof}

Some comments concerning characterization (\ref{item:val}) of the previous proposition are in order. Consider a convex set $Q$ containing a point $\bar{x}$, and let $\bar{v}\in N_Q(\bar{x})$ be a normal vector. Then the arguments $(\ref{item:part})\Rightarrow (\ref{item:val})$ and $(\ref{item:val})\Rightarrow (\ref{item:ident})$ show that a manifold $M\subset Q$ is identifiable at $\bar{x}$ for $\bar{v}$ if and only if the inclusion 
$$\bar{x}+\bar{v}\in \inter\Big(\bigcup_{x\in M} \big(x+N_Q(x)\big)\Big),$$ holds. The region $\bigcup_{x\in M} \big(x+N_Q(x)\big)$ is formed by attaching cones $N_Q(x)$ to each point $x\in  M$. This set is precisely the set of points in $\R^n$ whose projections onto $Q$ lie in $M$. Thus a manifold $M$ is identifiable at $\bar{x}$ for $\bar{v}$ whenever the region $\bigcup_{x\in M} \big(x+N_Q(x)\big)$ is ``valley-like'' around $\bar{x}+\bar{v}$. 
We end the section with an observation relating identifiable manifolds to critical cones.
\begin{cor}
Consider a closed set $Q\subset\R^n$ that is Clarke regular at a point $\bar{x}\in Q$. Suppose that $M\subset Q$ is a ${\bf C}^2$ identifiable manifold at $\bar{x}$ for $\bar{v}$. Then the critical cone $K_Q(\bar{x},\bar{v})$ coincides with the tangent space $T_M(\bar{x})$. 
\end{cor}
\begin{proof}
This is an immediate consequence of Proposition~\ref{prop:tang}. 
\end{proof}

\subsection{Projections onto identifiable manifolds}
In this subsection, we show that ${\bf C}^2$ identifiable manifolds occur precisely as the images of ${\bf C}^1$-smooth projections with constant rank. Predecessor results of a similar flavour can be found in \cite{Hare} and \cite{Wright}. We will need the following classical result. 

\begin{prop}\label{prop:exists}
Consider a ${\bf C}^2$-smooth manifold $M\subset\R^n$ and a point $\bar{x}\in M$. Then there exists a neighborhood $U$ of $\bar{x}$ so that 
the projection mapping $P_M$ is a ${\bf C}^1$-smooth submersion on $U$. In particular it sends open subsets of $U$ to open subsets of $M$.
\end{prop}
\begin{proof}
Let $V$ be a neighborhood of $\bar{x}$ so that $$M\cap V=\{x\in V:h_i(x)=0 \textrm{ for } i=1,\ldots,m\},$$
where $h_i\colon V\to\R$ are ${\bf C}^2$-smooth defining functions and the gradients $\nabla h_i(\bar{x})$ are linearly independent. 
Since $M$ is prox-regular at $\bar{x}$, the mapping $y\mapsto P_M(y)$ is single-valued near $\bar{x}$. Furthermore, the point $P_M(y)$ is precisely the point $x$ solving the system
\begin{align*}
h_i(x)&=0, \textrm{ for } i=1,\ldots,m,\\
x+\sum^m_{i=1}\lambda_i\nabla h_i(x) &=y.
\end{align*}
for some $\lambda\in\R^m$.
Applying the inverse function theorem to the system above yields the result.
\end{proof}

\begin{prop}[Smooth projections]\label{prop:main}
Consider a closed set $Q\subset\R^n$ and a ${\bf C}^2$ manifold $M\subset Q$ containing a point $\bar{x}$. Let $\bar{v}\in \hat{N}_Q(\bar{x})$ be a normal vector. Then the following are equivalent.
\begin{enumerate}
\item \label{item:ident_proj1} $M$ is an identifiable manifold around $\bar{x}$ for $\bar{v}$. 
\item \label{item:proj_ident1}
The set $Q$ is prox-regular at $\bar{x}$ for $\bar{v}$, and for all sufficiently small real numbers $\lambda$ and all sufficiently small neighborhoods $U$ of $\bar{x}+\lambda\bar{v}$,
\begin{itemize}
\item The projection mapping $P_Q$ is ${\bf C}^1$-smooth and has constant rank on the ball $U$. 
\item The image $P_Q(U)$ coincides with $M$ near $\bar{x}$.
\end{itemize}
\end{enumerate}
\end{prop}
\begin{proof}
$(\ref{item:ident_proj1})\Rightarrow (\ref{item:proj_ident1})$:
This follows immediately from Propositions~\ref{prop:prox},~\ref{prop:exists}, and \ref{prop:projid}. 

%

$(\ref{item:proj_ident1})\Rightarrow (\ref{item:ident_proj1})$:
This is immediate from Proposition~\ref{prop:proj_inter}.
\end{proof}

\section{The case of functions}\label{sec:func}
Many of the results we have derived in the previous sections were stated in terms of sets. Their generalizations to the functional setting can be carried out in a standard way by appealing to Proposition~\ref{prop:coherence}. In this section, we record some of these results for ease of reference in future work. Most of the proofs are omitted.

\begin{prop}[Prox-regularity of identifiable sets]
Consider a lsc function $f\colon\R^n\to\overline{\R}$ and a set $M\subset\R^n$ that is identifiable at $\bar{x}$ for $\bar{v}\in \hat{\partial} f(\bar{x})$. In addition, suppose that $f+\delta_M$ is prox-regular at $\bar{x}$ for $\bar{v}$. Then $f$ is prox-regular at $\bar{x}$ for $\bar{v}$.
\end{prop}

In stating the next results, it will be convenient to define ``subjets'' associated to each subdifferential. Namely, for a function $f\colon\R^n\to\overline{\R}$, we define the {\em limiting subjet} to be the set
$$[\partial f]:=\{(x,f(x),v)\in\R^n\times\R\times\R^n: v\in\partial f(x)\}.$$
Subjets corresponding to the other subdifferentials are defined analogously. These objects are useful for concisely describing variational properties of functions that are not necessarily subdifferentially continuous. On the other hand if $f$ is subdifferentially continuous, as is often the case, then in the results that follow we may simply use subdifferential graphs, instead of subjets.   

\begin{prop}[Reduction I]
Consider a lsc function $f\colon\R^n\to\overline{\R}$ and a set $M\subset\R^n$ so that $f+\delta_M$ is prox-regular at a point $\bar{x}$ for $\bar{v}\in\hat{\partial} f(\bar{x})$. Then $M$ is identifiable at $\bar{x}$ for $\bar{v}$ if and only if $$[\partial f]=[\partial (f+\delta_M)] ~\textrm{ locally around } (\bar{x},f(\bar{x}),\bar{v}).$$
\end{prop}

\begin{prop}[Reduction II]
Consider a lsc function $f\colon\R^n\to\overline{\R}$, a point $\bar{x}\in\R^n$, and a subgradient $\bar{v}\in \partial f(\bar{x})$. Suppose $$[\partial^P f]=[\partial f] \textrm{ locally around } (\bar{x},f(\bar{x}),\bar{v}),$$ and consider a set $M:=(\partial f)^{-1}(V)\cap \{x:|f(x)-f(\bar{x})|<\epsilon\}$, where $V$ is a convex, open neighborhood of $\bar{v}$ and $\epsilon >0$ is a real number.
Then the equation 
$$\gph [\partial f]=\gph [\partial (f+\delta_M)] \textrm{ holds locally around } (\bar{x},f(\bar{x}),\bar{v}).$$
\end{prop}

\begin{prop}[Prox-regularity and local minimality] {\ }{\\}
Consider a lsc function $f\colon\R^n\to\overline{\R}$ and a set $M\subset\R^n$ that is a locally minimal identifiable set at $\bar{x}$ for $\bar{v}\in \hat{\partial} f(\bar{x})$. Then $f$ is prox-regular at $\bar{x}$ for $\bar{v}$ if and only if $f+\delta_M$ is prox-regular at $\bar{x}$ for $\bar{v}$.
\end{prop}

The following is a standard generalization of critical cones to the functional setting.
\begin{defn}[Critical cones]
{\rm
For a function $f\colon\R^n\to\overline{\R}$ that is Clarke regular at a point $\bar{x}$, the {\em critical cone} of $f$ at $\bar{x}$ for $\bar{v}\in\partial f(\bar{x})$ is the set 
$$K_f(\bar{x},\bar{v}):=N_{\partial f(\bar{x})}(\bar{v}).$$ 
}
\end{defn}

\begin{defn}[Smooth derivability]
{\rm
A function $f\colon\R^n\to\overline{\R}$ is said to be {\em smoothly derivable} at a point $\bar{x}\in\R^n$ if $\gph f$ is smoothly derivable at $(\bar{x},f(\bar{x})$. 
}
\end{defn}

\begin{prop}[Critical cones more generally]
Consider a function $f\colon\R^n\to\overline{\R}$ that is Clarke regular at a point $\bar{x}$, and let $\bar{v}\in \partial f(\bar{x})$. Consider a nested sequence of open neighborhoods $V_i$ of $\bar{v}$ satisfying $\bigcap_{i=1}^{\infty} V_i=\{\bar{v}\}$ and a sequence of reals $\epsilon_i\downarrow 0$, and the corresponding preimages $M_i:=\partial f^{-1}(V_i)\cap \{x:|f(x)-f(\bar{x})|<\epsilon\}$. Assume that each $f+\delta_{M_i}$ is smoothly derivable at $\bar{x}$. Then the inclusion 
\begin{equation}
K_f(\bar{x},\bar{v})\supset \clco \bigcap_{i=1}^{\infty} T_{M_i}(\bar{x}),
\end{equation}
holds.
Assume in addition that each $f+\delta_{M_i}$ is prox-regular at $\bar{x}$ for $\bar{v}$ and that the formula 
\begin{equation}
\clco\bigcap_{i=1}^{\infty} T_{M_i}(\bar{x})=\bigcap_{i=1}^{\infty} \clco T_{M_i}(\bar{x}),
\end{equation}
holds. Then each $M_i$ is an identifiable set at $\bar{x}$ for $\bar{v}$ and we have 
$$K_f(\bar{x},\bar{v})=\clco\bigcap_{i=1}^{\infty} T_{M_i}(\bar{x}).$$
\end{prop}
In particular, the following is true.

\begin{prop}[Critical cones as tangential approximations]{\ }{\\}
Consider a function $f\colon\R^n\to\overline{\R}$ that is Clarke regular at a point $\bar{x}$, and let $M$ be a locally minimal identifiable set at $\bar{x}$ for $\bar{v}\in \partial f(\bar{x})$. Suppose furthermore that $f+\delta_M$ is prox-regular at $\bar{x}$ for $\bar{v}$ and is smoothly derivable at $\bar{x}$. Then the equation 
$$\clco T_M(\bar{x})= K_f(\bar{x},\bar{v}),$$
holds.
\end{prop}

\begin{defn}[Identifiable manifolds]
{\rm
Given a function $f\colon\R^n\to\overline{\R}$, a set $M\subset\R^n$ is a ${\bf C}^p$ {\em identifiable manifold} at $\bar{x}$ for $\bar{v}\in\partial f(\bar{x})$ provided that the following hold.
\begin{itemize}
\item $M$ is a ${\bf C}^p$ manifold around $\bar{x}$ and the restriction $f\big|_M$ is ${\bf C}^p$-smooth near $\bar{x}$.
\item $M$ is identifiable at $\bar{x}$ for $\bar{v}$.
\end{itemize}
}
\end{defn}

\begin{prop}
Consider a lsc function $f\colon\R^n\to\overline{\R}$ and a ${\bf C}^2$ identifiable manifold $M\subset \dom f$ at $\bar{x}$ for $\bar{v}\in\hat{\partial} f(\bar{x})$. Then $M$ is automatically a locally minimal identifiable set at $x\in M$ for $v\in N_Q(x)$, whenever the triple $(x,f(x),v)$ is near 
$(\bar{x},f(\bar{x}),\bar{v})$. 
\end{prop}

A functional version of partial smoothness is as follows. For a convex sets $K\subset\R^n$, we let $\para K$ be parallel subspace to $K$, that is $\para K=\spann \{K-x\}$, where $x$ is any point of $K$.

\begin{defn}[Directional partial smoothness]
{\rm 
Consider a lsc function $f\colon\R^n\to\overline{\R}$ and a ${\bf C}^2$-manifold $M\subset\dom f$. Then $f$
is {\em partly smooth with respect to} $M$ at $\bar{x}\in M$ {\em for} $\bar{v}\in\partial f(\bar{x})$ if
\begin{enumerate}
\item {\bf (prox-regularity)} $f$ is prox-regular at $\bar{x}$ for $\bar{v}$. 
\item {\bf (smoothness)} The restriction $f\big|_M$ is ${\bf C}^2$ smooth around $\bar{x}$. 
\item {\bf (sharpness)} $\para \hat{\partial} f =N_M(\bar{x})$.
\item {\bf (continuity)} There exists a neighborhood $V$ of $\bar{v}$, such that the mapping, $x\mapsto V\cap\partial f(x)$, when restricted to $M$, is inner-semicontinuous at $\bar{x}$.  
\end{enumerate}
}
\end{defn}

\begin{prop}[Identifiable manifolds and partial smoothness]{\ }{\\}
Consider a lsc function $f\colon\R^n\to\overline{\R}$ and a set $M\subset\R^n$, so that $M$ is a ${\bf C}^2$ manifold around $\bar{x}$ and $f$ is ${\bf C}^2$-smooth on $M$ around $\bar{x}$. Consider a subgradient $\bar{v}\in\hat{\partial} f(\bar{x})$. Then the following are equivalent.
\begin{enumerate}
\item $M$ is an identifiable manifold at $\bar{x}$ for $\bar{v}$. 
\item  We have 
$$[\partial f]=[\partial (f+\delta_M)] ~\textrm{ locally around } (\bar{x}, f(\bar{x}),\bar{v}).$$
\item
\begin{itemize}
\item $f$ is partly smooth with respect to $M$ at $\bar{x}$ for $\bar{v}$. 
\item the strong inclusion $\bar{v}\in \ri \hat{\partial} f(\bar{x})$ holds. 
\end{itemize}
\end{enumerate}
\end{prop}

Given a function $f\colon\R^n\to\overline{\R}$ and a parameter $\lambda>0$, the {\em proximal mapping} $P_{\lambda}f\colon\R^n\to\overline{\R}$ is defined by
$$P_{\lambda}f(x):=\argmin\big\{f(y)+\frac{1}{2\lambda}|y-x|^2 \big\}.$$ 
The results of Section~\ref{sec:geo}, concerning projections, may be extended to the functional setting with the proximal mapping taking the place of the metric projection. However for the sake of brevity, we do not pursue this further.

\begin{acknowledgement}
The authors thank Alex D. Ioffe for the insightful conversations related to the current work.
\end{acknowledgement}


\bibliographystyle{plain}
\small
\parsep 0pt
\bibliography{dim_graph}

%
%

\end{document}